\documentclass{siamart0216}
\newsiamremark{remark}{Remark}

\usepackage{amsfonts,amsmath,amssymb}
\usepackage{mathrsfs,mathtools,stmaryrd,wasysym}
\usepackage{enumitem}
  \setlist[itemize]{leftmargin=*}
  \setlist[enumerate]{leftmargin=*}
\usepackage{xspace,mydef}
\usepackage{esint}
\usepackage{tikz}
  \usetikzlibrary{cd}
  \tikzcdset{row sep/normal=1.5cm}
  \tikzcdset{column sep/normal=2.5cm}

\usepackage{hyperref}
\DeclareMathAlphabet{\mathpzc}{OT1}{pzc}{m}{it}

\newcommand{\TheTitle}{Asymptotic compatibility of parametrized optimal design problems}
\newcommand{\ShortTitle}{Asymptotically Compatible Schemes}
\newcommand{\TheAuthors}{T.~Mengesha, A.J.~Salgado, J.M.~Siktar}

\headers{\ShortTitle}{\TheAuthors}

\title{\TheTitle}

\author{Tadele Mengesha\thanks{Department of Mathematics, University of Tennessee, Knoxville, TN 37996, USA. \\
    (\email{megnesha@utk.edu}, \url{https://math.utk.edu/people/tadele-mengesha/})}
  \and
  Abner J.~Salgado\thanks{Department of Mathematics, University of Tennessee, Knoxville, TN 37996, USA. \\
    (\email{asalgad1@utk.edu}, \url{https://math.utk.edu/people/abner-salgado/})}
  \and
  Joshua M.~Siktar\thanks{Department of Mathematics, University of Tennessee, Knoxville, TN 37996, USA. \\
  \emph{Current address:} Department of Mathematics, Texas A\&M University, College Station, TX 77843, USA. \\
    (\email{jmsiktar@tamu.edu}, \url{https://www.math.tamu.edu/directory/formalpg.php?user=jmsiktar})}
}

\ifpdf
\hypersetup{
  pdftitle={\TheTitle},
  pdfauthor={\TheAuthors}
}
\fi

\begin{document}

\maketitle

\everymath{\displaystyle}

\begin{abstract}
  We study optimal design problems where the design corresponds to a coefficient in the principal part of the state equation. The state equation, in addition, is parameter dependent, and we allow it to change type in the limit of this (modeling) parameter. We develop a framework that guarantees asymptotic compatibility, that is unconditional convergence with respect to modeling and discretization parameters to the solution of the corresponding limiting problems. This framework is then applied to two distinct classes of problems where the modeling parameter represents the degree of nonlocality. Specifically, we show unconditional convergence of optimal design problems when the state equation is either a scalar-valued fractional equation, or a strongly coupled system of nonlocal equations derived from the bond-based model of peridynamics.  
\end{abstract}



\section{Introduction}
\label{sec:Intro}
This work analyzes the approximation of solutions $(\mathfrak{a}, u)$ of parametrized optimal design problems of  the form
\begin{equation}\label{GenericOptDesign}
  \min \{ J(\mathfrak{a}, u) \ | \ (\mathfrak{a}, u) \in \mathcal{H} \times X_{\delta} \},
\end{equation}
subject to the constraint
\begin{equation}\label{GenericOptDesignb}
  \frakB_\delta[\fraka](u,v) = \langle f, v \rangle, \ \ \forall v \in X_\delta. 
\end{equation}
The specific forms of the design space $ \mathcal{H}$, the state space $X_\delta$, and the objective  functional $J$ will be introduced later. The function $f$ is taken to be a fixed data. For a given $\mathfrak{a} \in \mathcal{H}$  and the modeling parameter $\delta\geq 0$,  $\frakB_\delta[\fraka]$  is a bilinear form on $X_\delta$.  The parametrized bilinear forms $\frakB_\delta[\fraka]$ have an asymptotic continuity with respect to the modeling parameter. In other words, there is a given bilinear form $\frakB_0[\fraka]$ such that, as $\delta \to 0$,  $\frakB_\delta[\fraka] \to \frakB_0[\fraka]$, in some appropriate sense. 

In this paper, we provide sufficient conditions for \eqref{GenericOptDesign}---\eqref{GenericOptDesignb} to have a solution corresponding to each $\delta.$ In addition, we present an approximation framework that not only gives an approximation to solutions for each $\delta$ but also remains compatible with the asymptotic continuity property of the bilinear forms. That is, approximate solutions that depend of model parameter $\delta$ and the discrete parameter converge, unconditionally, to a solution of the limiting optimal design problem. 

Approximation schemes of the above type, called \emph{asymptotically compatible (AC) schemes},  for solutions of parametrized problems have become a focus of research in recent years,  especially in connection with nonlocal models of diffusion as well as mechanics, where the parameter $\delta$ represents the degree of nonlocality.  In this case, the main question of interest is, in the event of vanishing nonlocality, i.e. $\delta\to 0$, whether the approximate solutions converge unconditionally  to the solution of the corresponding local equation. 
Reference \cite{MR3231986} was the first to rigorously  analyze AC schemes for linear equations in nonlocal conductivity and systems of equations derived from the state-based model of peridynamics. Other papers studying the asymptotic compatibility of scalar-valued nonlocal equations include \cite{du2018asymptotically, leng2021asymptotically, tao2017nonlocal, tian2016asymptotically, tian2015nonconforming, glusa2023asymptotically, du2024asymptoticallycompatibleschemesnonlinear}; those concerning peridynamic models include \cite{leng2020asymptotically, tao2017nonlocal}; and the Nonlocal Ohta--Kawasaki model is discussed in \cite{luo2024asymptotically}. These works show how critical the judicious choice of discretization is. Otherwise, approximate solutions may only be conditionally convergent to the solution of the corresponding local, continuous problem  \cite{tian2013analysis, MR4064535}. 

The current work, motivated by the above papers, aims at extending the analysis of the AC framework to parametrized optimal design problems. The paper has two parts. The first part introduces  the abstract AC framework for optimal design problems; that is, we provide sufficient conditions that demonstrate unconditional convergence of approximate solutions with respect to a modeling and a discretization parameter. The second part applies the developed framework to two different classes of nonlocal optimal design problems, where the modeling parameter is the degree of nonlocality. To our knowledge, this is the first paper to discuss AC schemes for optimal design problems.

The development of the abstract asymptotic compatibility framework for optimal design problems essentially follows what has been done in \cite{MR3231986,  du2024asymptoticallycompatibleschemesnonlinear} while paying special attention to the nonlinearity of our problem. Indeed, the constraint equation asserts a nonlinear relationship between a choice of material represented by $\mathfrak{a}$, and the resulting state $u$.

The first class of problems of the form \eqref{GenericOptDesign}---\eqref{GenericOptDesignb} that we apply our theory to is that of nonlocal optimal design problems modeling conductivity. The motivation here is to optimally distribute a (set of) materials within a given domain to achieve a desired conductivity property; see \cite{evgrafov2020nonlocal, evgrafov2019sensitivity}. The second class of  problems comes from nonlocal mechanics, where we look for an optimal constitutive material distribution that gives a globally rigid material. Here we use the linearized bond-based model of peridynamics (PD) \cite{silling2000reformulation, silling2007peridynamic} as the state equation; its solutions are vector-valued displacement fields. The PD model was introduced to model physical phenomena with inherent discontinuities, such as the formation of cracks in solids \cite{silling2005peridynamic, silling2010crack, silling2014peridynamic, du2020nonlocal, diehl2022comparative}. In both of these model problems, the design space will be a subset of  the set of positive and bounded functions, which will serve as coefficients in the bilinear form and represent material properties.   

A key step for studying existence of solutions to these optimal design problems is to select a single notion of convergence for coefficients shared between the nonlocal and local problems. For local optimal design problems, the classical notions of $G$- and $H$-convergence play an important role in studying existence of solutions. The reader is referred to \cite{tartar1979compensated, tartar1995remarks, allaire2001some, spagnolo1968sulla, spagnolo1976convergence, spagnolo2007sulla, tartar1975problemes, tartar1979estimations, tartar1976quelques, simon1979g, zhikov1981g, jikov2012homogenization} for classical results, and to \cite{deckelnick2011identification, deckelnick2012convergence, hinze2016matrix} for applications to local optimal design problems. Meanwhile, for nonlocal equations, analogous results are proven in \cite{bonder2017h, bellido2021simple, kassmann2019homogenization, munoz2020elementary}, under the assumption that the design coefficients converge weak-*  in $L^{\infty}$. In addition, in these references it is shown in \cite{jikov2012homogenization, allaire2001some} that the weak-* limit does not coincide with the classical homogenized limit.

An alternative approach is to use a smoother class of coefficients, one where all sequences have uniformly convergent sub-sequences. This idea was explored for nonlocal problems in \cite{andres2017convergence, andres2021optimal}. However, assuming higher regularity of the coefficients is not physically realistic for many applications.  As a result, this paper uses a third approach akin to the one used in \cite{andres2015nonlocal, andres2023minimization}: we use a compliance cost functional to avoid the issues with $G$-convergence, while still using a rougher class of coefficients where weak-* $L^\infty$ convergence is the natural choice, for both the nonlocal and local problems. 

In addition, we will approximate solutions to these optimal design problems using the finite element method. References \cite{bonito2018numerical, borthagaray2019weighted, acosta2017fractional, d2019priori, MR3702417} study the finite element approximations of solutions to nonlocal equations, while \cite{deckelnick2012convergence, deckelnick2011identification, hinze2016matrix} provide relevant techniques specific to the approximation of local optimal design problems. These references also provide numerical illustrations.  We also provide numerical results for the aforementioned nonlocal conductivity problem, and detail the utilized projected gradient descent algorithm. To our knowledge, this is the first work to include numerical results for any nonlocal optimal design problem.  We finally mention that there is already a wide literature of works that consider implementation of finite element schemes for solving the scalar-valued fractional Laplacian equation and its variants, such as \cite{acosta2017fractional, acosta2017short, ainsworth2018towards, dohr2019fem, feist2023fractional, gimperlein2019corner, lei2023finite, zhou2023adaptive}. The paper \cite{d2019priori} also gives numerical results for fractional optimal control problems.

Let us outline the contents of our the paper. The abstract asymptotic compatibility framework is developed in Section \ref{sec:AbstractSetting}. Section \ref{sec:Conductivity} gives the first application of this framework, to a family of problems from nonlocal conductivity. 
Section \ref{sec:Peridynamics} then gives the second application of our asymptotic compatibility framework, to a family of problems based on the bond-based model of peridynamics. Finally, we provide some numerical experiments for the nonlocal conductivity problem in Section \ref{sec:Numerics}.

We close this introduction by describing notations we will be using throughout this work. The relation $A \lesssim B$ shall stand for $A \leq c B$ for a nonessential constant $c$ whose value may change at each occurrence. $A \eqsim B$ is short for $A \lesssim B \lesssim A$. The symbol $\hookrightarrow$ denotes a continuous embedding between topological vector spaces.

\section{A general framework}
\label{sec:AbstractSetting}

Let us describe the general framework of our family of optimal design problems.

\subsection{Parametrized optimal design problems}
We make the following assumptions.
\begin{itemize}
  \item The \emph{design space} is a compact metric space $(\calH, d)$.

  \item Let $X_0$ and $X_1$ be Hilbert spaces with $X_0$ compactly embedded and dense in $X_1$. The parametrized \emph{state spaces} comprise a family of intermediate and nested Hilbert spaces $\{X_\delta\}_{\delta \in [0,1]}$, which satisfy
  \[
    X_{\delta_1} \hookrightarrow X_{\delta_2}, \qquad \forall \delta_1, \delta_2 \in [0,1] \ : \ \delta_1 < \delta_2,
  \]
  with an embedding constant independent of $\delta_1$ and $\delta_2$. The inner product in $X_1$ shall be denoted by $\langle \cdot, \cdot \rangle$.

  \item The parametrized \emph{state equations} are described by a given, and fixed, $f \in X_1$ and a family of mappings $\{\frakB_\delta\}_{\delta \in [0,1)}$ with $\frakB_\delta : \calH \to \calL_2^{\sym}(X_\delta)$, where $\calL_2^{\sym}(Y)$ denotes the space of bounded and symmetric bilinear forms on $Y$. Indeed, for each $\delta \in [0,1)$, the state equation is as follows. Assume that $\fraka \in \calH$ is given. Find $u \in X_\delta$ such that
  \begin{equation}
  \label{eq:StateDeltaAbs}
    \frakB_\delta[\fraka](u,v) = \langle f, v \rangle, \qquad \forall v \in X_\delta.
  \end{equation}
  Owing to symmetry, this is equivalent to the minimization of the following energy:
  \begin{equation}
  \label{eq:EnergyDeltaAbs}
    E_\delta[\fraka](w) = \frac12 \frakB_\delta(\fraka)(w,w) - \langle f, w \rangle.
  \end{equation}

  \item For each $\delta \in [0,1)$ the \emph{admissible set} is
  \begin{equation}
  \label{eq:AdmissSetDeltaAbs}
    \calZ^\delta = \left\{ (\fraka,u) \in \calH \times X_\delta \ \middle| \ u \in \argmin_{w \in X_\delta} E_\delta[\fraka](w) \right\}.
  \end{equation}
  Notice that, at this stage, nothing is preventing $\calZ^\delta$ from being empty.

  \item The \emph{objective} $J: \calH \times X_1 \to \R$ shall be of regularized compliance type. Namely, we assume we have at hand $\phi : \calH \to \R$ that is non-negative, convex, and lower semicontinuous. Then, for all $\frakb \in \calH$ and $w \in X_1$, we have
  \begin{equation}
  \label{eq:ObjectiveAbs}
    J(\frakb,w) = \langle f, w \rangle + \phi(\frakb).
  \end{equation}
\end{itemize}

With this notation at hand we can describe the class of abstract parametrized optimal design problems that we are interested in. For each $\delta \in [0,1)$ we seek
\begin{equation}
\label{eq:OptDesignDeltaAbs}
  (\ofraka_\delta, \ou_\delta) \in \argmin \left\{ J(\fraka,u) \ \middle| \ (\fraka,u) \in  \calZ^\delta \right\}.
\end{equation}

\subsection{Structural assumptions}
\label{sub:AssumeAbstract}

We begin with a series of assumptions that will be used not only to prove existence of solutions to our parametrized optimal design problems, but will also be used to study their asymptotic compatibility. First, the family $\{X_\delta \}_{\delta \in [0,1]}$ satisfies:
\begin{itemize}
  \item \textbf{Uniform embedding:} There is $M>0$ such that for every $\delta \in [0,1)$ and all $v \in X_\delta$ we have
  \begin{equation}
  \label{eq:S1}
  \tag{S1}
    \| v \|_{X_1} \leq M \| v \|_{X_\delta}.
  \end{equation}

  \item \textbf{Asymptotic compactness:} Let the family $\{ v_\delta \}_{\delta \in (0,1)}$ be such that $v_\delta \in X_\delta$ and satisfies
  \begin{equation}
  \label{eq:S2}
  \tag{S2}
    \sup_{\delta \downarrow 0} \| v _\delta \|_{X_\delta} < \infty,
  \end{equation}
  then, $\{ v_\delta \}_{\delta \in (0,1)}$ is relatively compact in $X_1$, and any limit point belongs to $X_0$.
\end{itemize}

In addition, the family of bilinear forms $\{\frakB_\delta\}_{\delta \in [0,1)}$ satisfies:
\begin{itemize}
  \item \textbf{Strong continuity:} For every $\delta \in [0,1)$ the mapping $\frakB_\delta$ is strongly continuous, \ie if $\{\fraka_j\}_{j =1}^\infty \subset \calH$ is such that $\fraka_j \overset{d}{\to} \fraka \in \calH$, then, for every $v \in X_\delta$, we have
  \begin{equation}
  \label{eq:E1}
  \tag{B1}
    \lim_{j \to \infty} \frakB_\delta[\fraka_j](v,v) = \frakB_\delta[\fraka][v,v].
  \end{equation}

  \item \textbf{Uniform boudnedness:} There is $A>0$ such that, for every $\delta \in [0,1)$ and all $\fraka \in \calH$, we have
  \begin{equation}
  \label{eq:E2}
  \tag{B2}
    \left| \frakB_\delta[\fraka](v,v) \right| \leq A \| v \|_{X_\delta}^2, \qquad \forall v \in X_\delta.
  \end{equation}

  \item \textbf{Uniform coercivity:} There is $\alpha > 0$ such that, for all $\delta \in [0,1)$ and all $\fraka \in \calH$, we have
  \begin{equation}
  \label{eq:E3}
  \tag{B3}
    \alpha \| v \|_{X_\delta}^2 \leq \frakB_\delta[\fraka](v,v), \qquad \forall v \in X_\delta.
  \end{equation}

  \item \textbf{Parametric continuity:} Assume that $\{\fraka_\delta\}_{\delta \in (0,1)} \subset \calH$ and $\fraka \in \calH$ are such that, as $\delta \downarrow 0$, we have $\fraka_\delta \overset{d}{\to} \fraka$. Then,
  \begin{equation}
  \label{eq:2.22}
  \tag{B4}
    \lim_{\delta \downarrow 0} \left( \frakB_\delta[\fraka_\delta](v,w) - \frakB_\delta[\fraka](v,w) \right) = 0, \qquad \forall v, w \in X_0.
  \end{equation}
  Moreover,
  \begin{equation}
  \label{eq:2.23}
  \tag{B5}
    \lim_{\delta \downarrow 0} \frakB_\delta[\fraka](v,w) = \frakB_0[\fraka](v,w) , \qquad \forall v, w \in X_0.
  \end{equation}

  \item \textbf{Lower semicontinuity:} Let $\fraka \in \calH$. Assume that $\{v_\delta  \in X_\delta \}_{\delta \in (0,1)}$ is such that there is $v \in X_0$ for which, as $\delta \downarrow 0$, we have $v_\delta \to v$ in $X_1$. Then,
  \begin{equation}
  \label{eq:2.24}
  \tag{B6}
    \frakB_0[\fraka](v,v) \leq \liminf_{\delta \downarrow 0} \frakB_\delta[\fraka](v_\delta,v_\delta).
  \end{equation}
\end{itemize}

\subsection{Existence of minimizers}
\label{sub:Analysis}
We now analyze the parametrized optimal design problems \eqref{eq:OptDesignDeltaAbs}. To begin we notice that, owing to the given assumptions, the state equations are well-posed and, more importantly, the solutions are bounded uniformly with respect to the parameter $\delta$.

\begin{lemma}[uniform well posedness]
\label{lem:StatesAreOKAbs}
For every $\delta \in [0, 1)$ and all $\fraka \in \calH$ there is a unique $u_\delta \in X_\delta$ that solves \eqref{eq:StateDeltaAbs}. In addition, this solution satisfies
\[
  \| u_\delta \|_{X_\delta} \leq \frac{M}\alpha \| f\|_{X_1},
\]
and it is uniquely characterized by the optimality condition
\begin{equation}
  \label{eq:OptCondStateAbs}
    E_\delta[\fraka](u_\delta) = -\frac12 \frakB_\delta[\fraka](u_\delta,u_\delta).
\end{equation}
\end{lemma}
\begin{proof}
Owing to \eqref{eq:E2} and \eqref{eq:E3}, existence and uniqueness follows from the Lax--Milgram lemma. The claimed estimate follows from \eqref{eq:S1} and \eqref{eq:E3} by setting $v = u_\delta$ in \eqref{eq:StateDeltaAbs}. Finally, the optimality condition \eqref{eq:OptCondStateAbs} is obtained by setting, in \eqref{eq:StateDeltaAbs}, $v = u_\delta$ and using the definition of the energy $E_\delta[\fraka]$.
\end{proof}

\begin{remark}[notation]
Notice that the previous result implies that, for every $\delta \in [0,1)$, the set $\calZ^\delta$ is not only nonempty, but it is actually the graph of a mapping $\calH \to X_\delta$. We shall denote this mapping by $T^\delta$. In addition, we define the \emph{reduced cost} (objective) to be, for every $\delta \in [0,1)$,
\begin{equation}
\label{eq:ReducedCostAbs}
  r^\delta(\fraka) = J(\fraka,T^\delta(\fraka)).
\end{equation}

\end{remark}

We are now ready to show existence of solutions to our parametrized optimal design problems \eqref{eq:OptDesignDeltaAbs}.

\begin{theorem}[existence]\label{thm:ExistenceAbs} Under the stated assumptions, for every $\delta \in [0,1)$, problem \eqref{eq:OptDesignDeltaAbs} has a solution.
\end{theorem}
\begin{proof}
We begin by observing that, owing to the uniform estimate in Lemma~\ref{lem:StatesAreOKAbs} and the fact that $\phi$ is non-negative, the objective $J$ is bounded from below in $\calZ^\delta$. Let then $\{(\fraka_j,u_j)\}_{j=1}^\infty \subset \calZ^\delta$ be an ``infimizing'' sequence. The compactness of $\calH$ implies that we can extract a (non-relabeled) sub-sequence $\{\fraka_j\}_{j=1}^\infty$ and $\ofraka_\delta \in \calH$ such that $\fraka_j \overset{d}{\to} \ofraka_\delta$ as $j \uparrow \infty$. Let $\ou_\delta = T^\delta(\ofraka_\delta)$ so that, by construction, $(\ofraka_\delta, \ou_\delta) \in \calZ^\delta$. The goal shall be to show that this is a minimizer.

Owing to \eqref{eq:OptCondStateAbs} we have, for every $j \in \N$,
\[
  E_\delta[\fraka_j](u_j) = -\frac12 \frakB[\fraka_j](u_j,u_j).
\]
Then, since $u_j$ minimizes $E_\delta[\fraka_j]$, we have
\[
  -\frac12 \frakB[\fraka_j](u_j,u_j) \leq E_\delta[\fraka_j](\ou_\delta) = \frac12 \frakB_\delta[\fraka_j](\ou_\delta,\ou_\delta) - \langle f, \ou_\delta \rangle.
\]
Passing to the limit $j \uparrow \infty$, and using the strong continuity \eqref{eq:E1} we see that, since $\fraka_j \overset{d}{\to} \ofraka_\delta$,
\[
  \limsup_{j \to \infty} -\frac12 \frakB[\fraka_j](u_j,u_j) \leq \frac12 \frakB_\delta[{\ofraka_{\delta}}](\ou_\delta,\ou_\delta) - \langle f, \ou_\delta \rangle = -\frac12 \frakB_\delta[{\ofraka_\delta}](\ou_\delta,\ou_\delta),
\]
where we again used the optimality condition \eqref{eq:OptCondStateAbs}. Using now the state equation \eqref{eq:StateDeltaAbs} we infer that
\[
  \langle f , \ou_\delta \rangle \leq \liminf_{j \to \infty} \langle f , u_j \rangle.
\]
Next, the lower semicontinuity of $\phi$ implies
\[
  \phi(\ofraka_\delta) \leq \lim_{j \to \infty} \phi(\fraka_j).
\]
Finally, we combine the previous two inequalities to see that
\[
  J(\ofraka_\delta,\ou_\delta) \leq \lim_{j \to \infty} J(\fraka_j,u_j) = \inf\left\{ J(\fraka,u) \ \middle| \ (\fraka,u) \in \calZ^\delta \right\},
\]
and this completes the proof.
\end{proof}

\begin{remark}[lack of uniqueness]
Notice that we are not claiming any sort of uniqueness of solutions to \eqref{eq:OptDesignDeltaAbs}. This is due to the fact that $\calZ^\delta$ is not convex.
\end{remark}

\subsection{Variational convergence as $\delta \downarrow 0$}
\label{sub:DeltaToZero}

Let us now study the passage to the limit $\delta \downarrow 0$. Our goal is to show that if $\{(\ofraka_\delta,\ou_\delta) \in \calH \times X_\delta\}_{\delta \in (0,1)}$ is a family of solutions to problem \eqref{eq:OptDesignDeltaAbs} then, as $\delta \downarrow 0$, we may pass to the limit and obtain $(\ofraka_0,\ou_0)$, which is a solution to \eqref{eq:OptDesignDeltaAbs} for $\delta = 0$. The tool that we shall use for this is called $\Gamma$-convergence; see \cite{MR1968440,MR1201152}.

\begin{theorem}[$\Gamma$-convergence]
\label{thm:GconvAsDeltaTo0}
Under our structural assumptions we have that the reduced cost functionals $\{r^\delta\}_{\delta \in (0,1)}$, defined in \eqref{eq:ReducedCostAbs}, $\Gamma$-converge to $r^0$ with respect to the metric $d$. Moreover, this family is equicoercive.
\end{theorem}
\begin{proof}
Equicoercivity  of the family of functionals follows by definition from the assumption tha $\calH$ is compact. 
The sequence $\{r^\delta\}_{\delta \in (0,1)}$ $\Gamma$-converges to $r^0$ is equivalent to proving \emph{lim-inf and lim-sup inequalities}. We prove these inequalities separately:
\begin{itemize}
  \item \textbf{lim-inf inequality:} Assume that $\{\fraka_\delta \}_{\delta \in (0,1)} \subset \calH$ is such that there is $\fraka \in \calH$ for which, as $\delta \downarrow0$, $\fraka_\delta \overset{d}{\to} \fraka$. 
  We prove that  $r^{0} (\fraka) \leq \liminf_{\delta \downarrow 0} r^{\delta}(\fraka_\delta)$. Let $u_\delta = T^\delta (\fraka_\delta)$ and $u=T^0(\fraka)$. Owing to Lemma~\ref{lem:StatesAreOKAbs} we have
  \[
    E_\delta[\fraka_\delta](u_\delta) =- \frac12 \frakB_\delta[\fraka_\delta](u_\delta,u_\delta), \qquad E_0[\fraka](u) =- \frac12 \frakB_0[\fraka](u,u),
  \]
  and, since for all $\delta \in (0,1)$ $X_0 \subset X_\delta$, we conclude that
  \[
    \lim_{\delta \downarrow 0} \left( - \frac12 \frakB_\delta[\fraka_\delta](u_\delta,u_\delta) \right) \leq \lim_{\delta \downarrow 0} E_\delta[\fraka_\delta](u) = \frac12\lim_{\delta \downarrow 0} \frakB_\delta[a_\delta](u,u) - \langle f, u \rangle.
  \]
  We now combine \eqref{eq:2.22} and \eqref{eq:2.23} to conclude that
  \[
    \lim_{\delta \downarrow 0} \frakB_\delta[a_\delta](u,u) = \frakB_0[a](u,u).
  \]
  As a consequence, we have obtained that
  \[
    \langle f, u \rangle = \frakB_0[\fraka](u,u) \leq \lim_{\delta\downarrow0} \frakB_\delta[\fraka_\delta](u_\delta,u_\delta) = \liminf_{\delta \downarrow 0} \langle f, u_\delta \rangle,
  \]
  where we also used that $u_\delta = T^\delta (\fraka_\delta)$ and $u = T^0(\fraka)$, respectively. Next, the lower semicontinuity of $\phi$ immediately implies
  \[
    \phi(\fraka) \leq \liminf_{\delta \downarrow0}\phi(\fraka_\delta).
  \]
  Finally, the last two inequalities, together with the super-additivity of the limit inferior, readily imply that
  \[
    r^0(\fraka) \leq \liminf_{\delta \downarrow0} r^\delta(\fraka_\delta),
  \]
  as we intended to show.

  \item \textbf{Recovery sequence:} We use the constant recovery sequence. Let $\fraka \in \calH$, and we need to prove that
  \[
    \limsup_{\delta \downarrow0} r^\delta(\fraka) \leq r^0(\fraka).
  \]
  Set, for $\delta \in [0,1)$, $u_\delta = T^\delta(\fraka)$. Lemma~\ref{lem:StatesAreOKAbs} then shows that $\sup_{\delta \in (0,1)}\|u_\delta\|_{X_\delta} < \infty$ and, owing to \eqref{eq:S2}, there is $u \in X_0$ such that $u_\delta \to u$ in $X_1$. We now claim that $u=u_0$. To see this we must show that $E_0[\fraka](u) \leq E_0[\fraka](v)$ for all $v \in X_0$. Let $v \in X_0$ be arbitrary. Since, for any $\delta \in (0,1)$, we have $X_0 \hookrightarrow X_\delta$ we may conclude that
  \[
    E_\delta[\fraka](u_\delta) \leq E_\delta[\fraka](v).
  \]
  Use now \eqref{eq:2.24} and \eqref{eq:2.23}, respectively, on each side of this relation to see that
  \[
    E_0[\fraka](u) \leq \liminf_{\delta \downarrow0} E_\delta[\fraka](u_\delta) \leq \lim_{\delta \downarrow0} E_\delta[\fraka](v) = E_0[\fraka](v),
  \]
  and indeed $u=u_0$. From this we see that
  \[
    \lim_{\delta \downarrow0} r^\delta(\fraka) = \lim_{\delta \downarrow0} \left( \langle f, u_\delta \rangle +\phi(\fraka) \right) = \langle f,u_0 \rangle + \phi(\fraka) = \langle f,T^0(\fraka) \rangle + \phi(\fraka) = r^0(\fraka).
  \]
  Thus, the constant sequence is a recovery sequence.
\end{itemize}
\end{proof}

As a consequence of the previous result, we can prove convergence of minimizers (see \cite[Theorem 13.3]{MR3821514} and \cite[Corollary 7.20]{MR1201152}).

\begin{corollary}[convergence of minimizers]
\label{col:LimDelta0}
Let $\{\ofraka_\delta \}_{\delta \in (0,1)} \subset \calH$ be a family of optimal design coefficients, \ie for each $\delta > 0$,
\[
  \ofraka_\delta \in \argmin\left\{ r^\delta(\fraka) \ \middle| \ \fraka \in \calH \right\}.
\]
Assume that $\ofraka$ is a cluster point of $\{\ofraka_\delta\}_{\delta \in (0,1)}$, \ie $\ofraka$ is such that, up to sub-sequences, $\ofraka_\delta \overset{d}{\to} \ofraka$, as $\delta \downarrow0$. Then we have
\[
  \ofraka \in \argmin\left\{ r^0(\fraka) \ \middle| \ \fraka \in \calH \right\},
\]
and
\[
  \lim_{\delta\downarrow0} r^\delta(\ofraka_\delta) = r^0(\ofraka).
\]
Moreover, setting $\ou_\delta = T^\delta(\ofraka_\delta)$ and $\ou = T^0(\ofraka)$ we also obtain that
\begin{equation}
\label{eq:2.38}
  \lim_{\delta \downarrow0} \| \ou - \ou_\delta \|_{X_\delta} = 0.
\end{equation}
In addition $\ou_\delta \to \ou$ in $X_1$. Finally, $\phi(\ofraka_\delta) \to \phi(\ofraka)$ as $\delta \downarrow0$.
\end{corollary}
\begin{proof}
The existence of $\ofraka$ follows from the compactness of $\mathcal{H}$, while the convergence of  minima of $r^\delta$ is an immediate consequence of $\Gamma$-convergence and equicoercivity, see \cite[Theorem 13.3]{MR3821514} and \cite[Corollary 7.20]{MR1201152}. These properties were proved in Theorem~\ref{thm:GconvAsDeltaTo0}.

Next we prove the convergence of optimal states. Observe first that, owing to \eqref{eq:S1}, the convergence in $X_1$ will follow from \eqref{eq:2.38}. Next, the convergence of the reduced costs and the lower semicontinuity of $\phi$ imply
\[
  \limsup_{\delta \downarrow0 }\langle f, \ou_\delta \rangle = \limsup_{\delta \downarrow0 } \left[ r^\delta(\ofraka_\delta) - \phi(\ofraka_\delta) \right] \leq r^0(\ofraka) - \phi(\ofraka) = \langle f, \ou \rangle.
\]
We can also repeat the proof of the lim-inf inequality in Theorem~\ref{thm:GconvAsDeltaTo0} to obtain the reverse inequality and conclude that $\langle f, \ou_\delta \rangle \to \langle f, \ou \rangle$ as $\delta \downarrow0$. Notice that, in passing, we have shown the last claim, \ie
\[
  \phi(\ofraka_\delta) = r^\delta(\ofraka_\delta) - \langle f, \ou_\delta \rangle \to r^0(\ofraka_0) - \langle f, \ou \rangle = \phi(\ofraka).
\]

Let us now proceed with the proof of \eqref{eq:2.38}. Owing to the uniform coercivity \eqref{eq:E3} it suffices to show that
\[
  \lim_{\delta \downarrow0} \frakB_\delta[\ofraka_\delta](\ou - \ou_\delta,\ou - \ou_\delta ) = 0.
\]
We expand the bilinear form, use symmetry, and use the fact that $\ou_\delta$ solves the corresponding state equation to see that
\begin{align*}
  \frakB_\delta[\ofraka_\delta](\ou - \ou_\delta,\ou - \ou_\delta ) &= \frakB_\delta[\ofraka_\delta](\ou_\delta,\ou_\delta) -2\frakB_\delta[\ofraka_\delta](\ou_\delta,\ou) + \frakB_\delta[\ofraka_\delta](\ou,\ou) \\
    &= \langle f, \ou_\delta \rangle - 2 \langle f, \ou \rangle + \frakB_\delta[\ofraka_\delta](\ou,\ou).
\end{align*}
Using \eqref{eq:2.22} and \eqref{eq:2.23} we may pass to the limit and obtain the desired result.
\end{proof}

\subsection{Discretization}

We now introduce the discretization scheme for problem \eqref{eq:OptDesignDeltaAbs} via Galerkin-like techniques. To do so, fix $\delta \in [0,1)$ and assume that for every $h \in (0,h_0)$ we have at hand a finite dimensional $X_{\delta,h} \subset X_\delta$. In addition, for $h \in (0,h_0)$, $\calH_h$ is a closed subset of $\calH$.

The discrete optimal design problem is as follows. The discrete admissible set is
\[
  \calZ_h^\delta = \left\{ (\fraka_h,u_h) \in \calH_h \times X_{\delta,h} \ \middle| \ u_h \in \argmin_{w_h \in X_{\delta,h}} E_\delta[\fraka_h](w_h) \right\}.
\]
For each $\delta \in [0,1)$ and $h \in (0,h_0)$ we seek to find
\begin{equation}
\label{eq:DiscrOptDesignDeltaAbs}
  (\ofraka_{\delta,h}, \ou_{\delta,h}) \in \argmin\left\{ J(\fraka_h,u_h) \ \middle| \ (\fraka_h ,u_h) \in \calZ_h^\delta \right\}.
\end{equation}

\begin{remark}[variational vs.~full discretization]
We comment that, for reasons that will be clear once we deal with applications, we have chosen a full discretization of the design space, and not a variational one. The reader is referred to \cite{MR2122182} for an explanation of this terminology, to \cite{MR2977495,MR2997232,MR3759564} for examples of variational discretizations, and \cite{MR3207771,MR4196420} for full discretizations; mostly in the setting of coefficient identification for elliptic equations of second order.
\end{remark}

Without much effort we can show existence of discrete solutions. The following is the discrete analogue of Lemma~\ref{lem:StatesAreOKAbs}, and so we omit its proof.

\begin{lemma}[uniform well posedness]
\label{lem:DiscrStatesAreOKAbs}
For any $\delta \in [0,1]$, $h \in (0,h_0)$, and $\fraka \in \calH$ the energy $E_\delta[\fraka]$, defined in \eqref{eq:EnergyDeltaAbs} has a unique minimizer in $X_{\delta,h}$. If we denote this minimizer by $u_{\delta,h} \in X_{\delta,h}$ then we have the uniform bound
\[
  \| u_{\delta,h} \|_{X_\delta} \leq \frac{M}\alpha \| f \|_{X_1},
\]
and the characterizations
\[
  \frakB_\delta[\fraka](u_{\delta,h},v_h) = \langle f, v_h \rangle, \quad \forall v_h \in X_{\delta,h}, \qquad E_\delta[\fraka](u_{\delta,h}) = -\frac12 \frakB_\delta[\fraka](u_{\delta,h},u_{\delta,h}).
\]
\end{lemma}

\begin{remark}[notation]
Let us denote by $T_h^\delta : \calH \to X_{\delta,h}$ the discrete solution mapping for the state equation. Owing to Lemma~\ref{lem:DiscrStatesAreOKAbs} this is well-defined. The discrete admissible set can then be equivalently written as
\[
  \calZ_h^\delta = \left\{ (\fraka_h,u_h) \in \calH_h \times X_{\delta,h} \ \middle| \ u_h = T_h^\delta(\fraka_h) \right\}.
\]
The discrete reduced cost $r_h^\delta : \calH_\delta \to \R$ is then
\[
  r_h^\delta(\fraka_h) = J(\fraka_h,T_h^\delta \fraka_h).
\]
\end{remark}

Notice that, although we only need it for $\fraka_h \in \calH_h$, Lemma~\ref{lem:DiscrStatesAreOKAbs} shows that the discrete state equations are well posed for all values of the model and discretization parameters. The following result is the discrete analogue of Theorem~\ref{thm:ExistenceAbs}. Notice that, once again, we do not claim uniqueness.

\begin{theorem}[existence]
\label{thm:DiscrExistenceAbs}
For every $\delta \in [0,1)$ and $h \in (0,h_0)$ problem \eqref{eq:DiscrOptDesignDeltaAbs} has a solution.
\end{theorem}
\begin{proof}
The proof repeats that of Theorem~\ref{thm:ExistenceAbs}, but we use Lemma~\ref{lem:DiscrStatesAreOKAbs} and the fact that, for each $h \in (0,h_0)$, the set $\calH_h$ is closed.
\end{proof}

\subsubsection{Assumptions on discretization}
\label{subsub:AssumptionsOnDiscrAbs}

Our next goals are to pass to the limits $h \downarrow0$ but $\delta \in [0,1)$ fixed; $\delta \downarrow0$ but $h \in (0,h_0)$ fixed; and, most importantly, both $h \downarrow0$ and $\delta \downarrow0$, \ie asymptotic compatibility. To be able to make these limit passages, we shall require the following assumptions on the discretization scheme.

\begin{itemize}
  \item \textbf{Discrete embedding:} Fix $h \in (0,h_0)$. Then we have
  \begin{equation}
  \label{eq:DiscrEmbedding}
    X_{\delta_1,h} \hookrightarrow X_{\delta_2,h}, \qquad \forall \delta_1,\delta_2 \in [0,1) \ : \ \delta_1 < \delta_2,
  \end{equation}
  with an embedding constant independent of $h$, $\delta_1$, and $\delta_2$.

  \item \textbf{Approximation property:} Fix $\delta \in [0,1)$. For every $w \in X_\delta$ there is $\{w_h \in X_{\delta,h} \}_{h \in (0,h_0)}$ that satisfies
  \begin{equation}
    \label{eq:C1}
    \lim_{h \downarrow0} \| w - w_h \|_{X_{\delta}} = 0.
  \end{equation}

  \item \textbf{Density:} For each $h \in (0,h_0)$ there is a mapping $\Pi_h : \calH \to \calH_h$, such that for every $\fraka \in \calH$, we have
  \begin{equation}
  \label{eq:PhiAndProj}
    \phi(\Pi_h \fraka ) \to \phi(\fraka), \qquad h \downarrow0.
  \end{equation}
\end{itemize}

Before stating the next assumption, we put it in context. Fix $\delta \in [0,1)$ and let $\fraka \in \calH$ be arbitrary. For $h \in (0,h_0)$ the \emph{Galerkin projection} is the mapping $\calG_{\delta,h}[\fraka]: X_\delta \to X_{\delta,h}$ defined, for every $v \in X_\delta$, via
\[
  \frakB_\delta[\fraka](v-\calG_{\delta,h}[\fraka]v,w_h) = 0, \qquad \forall w_h \in X_{\delta,h}.
\]
Owing to \eqref{eq:C1} it is not difficult to see that, for all $v \in X_\delta$, $\calG_{\delta,h}[\fraka]v \to v$ in $X_{\delta}$ as $h \downarrow0$. The next assumption states that a somewhat \emph{perturbed} version of the Galerkin projection is also convergent.

\begin{itemize}
  \item \textbf{Perturbed Galerkin projection:} Fix $\delta \in [0,1)$ and let $\frakb \in \calH$. For $h \in (0,h_0)$ define the mapping $\hat\calG_{\delta,h}[\frakb]: X_\delta \to X_{\delta,h}$, for every $v \in X_\delta$, via
  \[
    \frakB_\delta[\Pi_h \frakb](\hat\calG_{\delta,h}[\frakb]v, w_h) = \frakB_\delta[\frakb](v, w_h), \qquad \forall w_h \in X_{\delta,h}.
  \]
  Then, for all $(\frakb,v) \in \calZ^\delta$, we have
  \begin{equation}
  \label{eq:Assumption2.16}
    \| v-\hat\calG_{\delta,h}[\frakb] v \|_{X_\delta} \to 0, \qquad h \downarrow0.
  \end{equation}
\end{itemize}

\subsubsection{Convergence as $h \downarrow0$}
\label{subsub:hToZero}

We are now ready to pass to the limit in the discretization parameter, \ie $h \downarrow 0$, while keeping the model parameter $\delta \in [0,1)$ fixed. The proof strategy in the following result loosely follows that used in \cite[Theorem 3.2]{MR2977495}.

\begin{theorem}[convergence as $h \downarrow0$]
\label{them:LimHtoZeroAbs}
Fix $\delta \in [0,1)$ and assume that the family of pairs $\{(\ofraka_h,\ou_h)\}_{h \in (0,h_0)}$ solves problem \eqref{eq:DiscrOptDesignDeltaAbs}. There is $(\ofraka,\ou) \in \calZ^\delta$ such that, as $h \downarrow0$, and up to sub-sequences, the following hold:
\begin{enumerate}
  \item $\ofraka_h \overset{d}{\to} \ofraka$.

  \item The pair $(\ofraka,\ou)$ solves problem \eqref{eq:OptDesignDeltaAbs}.

  \item $J(\ofraka_h,\ou_h) \to J(\ofraka,\ou)$.

  \item $\ou_h \to \ou$ in $X_\delta$.

  \item $\phi(\ofraka_h) \to \phi(\ofraka)$.
\end{enumerate}
\end{theorem}
\begin{proof}
We prove each statement in order of appearance.
\begin{enumerate}
  \item Since $\calH$ is compact, we can extract a (not relabeled) sub-sequence and $\ofraka \in \calH$, such that $\ofraka_h \overset{d}{\to} \ofraka$. Define $\ou = T^\delta(\ofraka)$ so that $(\ofraka,\ou) \in \calZ^\delta$.

  \item Since $(\ofraka_h,\ou_h) \in \calZ_h^\delta$ and $(\ofraka,\ou) \in \calZ^\delta$, we have
  \[
    E_\delta[\ofraka_h](\ou_h) = -\frac12 \frakB_\delta[\ofraka_h](\ou_h,\ou_h), \qquad E_\delta[\ofraka](\ou) = -\frac12 \frakB_\delta[\ofraka](\ou,\ou).
  \]
  \begin{multline*}
    -\frac12 \frakB_\delta[\ofraka_h](\ou_h,\ou_h) = E_\delta[\ofraka_h](\ou_h) \leq E_\delta[\ofraka_h](\calG_{\delta,h}[\ofraka]\ou) \\
      = \frac12 \frakB_\delta[\ofraka_h](\calG_{\delta,h}[\ofraka]\ou,\calG_{\delta,h}[\ofraka]\ou) - \langle f, \calG_{\delta,h}[\ofraka]\ou \rangle.
  \end{multline*}
  Observe now that, using \eqref{eq:E2}, \eqref{eq:E3} with the definition of $\calG_{\delta,h}[\ofraka]$, and the estimate in Lemma~\ref{lem:StatesAreOKAbs}, we get
  \begin{align*}
    \frakB_\delta[\ofraka_h](\calG_{\delta,h}[\ofraka]\ou,\calG_{\delta,h}[\ofraka]\ou) &= \frakB_\delta[\ofraka_h](\ou,\ou) + \frakB_\delta[\ofraka_h](\ou-\calG_{\delta,h}[\ofraka]\ou,\ou-\calG_{\delta,h}[\ofraka]\ou) \\
      &- 2 \frakB_\delta[\ofraka_h](\calG_{\delta,h}[\ofraka]\ou, \ou - \calG_{\delta,h}[\ofraka]\ou ) \\
    &\leq \frakB_\delta[\ofraka_h](\ou,\ou) + A \| \ou-\calG_{\delta,h}[\ofraka]\ou \|_{X_\delta}^2 \\
      &+ \frac{AM}{\alpha^2} \| f \|_{X_1} \| \ou-\calG_{\delta,h}[\ofraka]\ou \|_{X_\delta} \\
      &\to \frakB_\delta[\ofraka](\ou,\ou), \qquad h \downarrow0,
  \end{align*}
  where, to pass to the limit, we invoked \eqref{eq:E1}. Consequently,
  \[
    -\frac12 \limsup_{h \downarrow0}\frakB_\delta[\ofraka_h](\ou_h,\ou_h) \leq \frac12 \frakB_\delta[\ofraka](\ou,\ou) - \langle f, \ou \rangle = E_\delta[\ofraka](\ou) = -\frac12 \frakB_\delta[\ofraka](\ou,\ou),
  \]
  \ie
  \[
    \langle f, \ou \rangle = \frakB_\delta[\ofraka](\ou,\ou) \leq \liminf_{h \downarrow0} \frakB_\delta[\ofraka_h](\ou_h,\ou_h) = \liminf_{h \downarrow0} \langle f, \ou_h \rangle.
  \]
  This, combined with the lower semicontinuity of $\phi$ readily implies
  \begin{equation}
  \label{eq:2.60}
    J(\ofraka,\ou) \leq \liminf_{h \downarrow0}J(\ofraka_h,\ou_h).
  \end{equation}

  Let now $(\fraka,u) \in \calZ^\delta$ be arbitrary. We must show that $J(\ofraka,\ou) \leq J(\fraka,u)$. With this intention, note that $(\Pi_h \fraka, \hat\calG_{\delta,h}[\fraka]u) \in \calZ_h^\delta$ and thus we must have
  \[
    J(\ofraka_h,\ou_h) \leq J(\Pi_h \fraka, \hat\calG_{\delta,h}[\fraka]u).
  \]
  We now use assumptions \eqref{eq:PhiAndProj} and \eqref{eq:Assumption2.16} to pass to the limit and obtain
  \begin{multline*}
    J(\ofraka,\ou) \leq \liminf_{h \downarrow0}J(\ofraka_h,\ou_h) \leq \lim_{h \downarrow0}J(\Pi_h\fraka,\hat\calG_{\delta,h}[\fraka]u) \\
      = \lim_{h\downarrow0} \langle f, \hat\calG_{\delta,h}[\fraka] u \rangle + \lim_{h \downarrow0} \phi(\Pi_h \fraka) = \langle f, u \rangle + \phi(\fraka) = J(\fraka,u),
  \end{multline*}
  and the pair $(\ofraka,\ou)$ solves \eqref{eq:OptDesignDeltaAbs}.

  \item Notice that \eqref{eq:2.60} has already proven half of the assertion. The reverse inequality easily follows by recalling that $(\Pi_h \ofraka, \hat\calG_{\delta,h}[\ofraka]\ou) \in \calZ_h^\delta$ so that
  \[
    J(\ofraka_h,\ou_h) \leq J(\Pi_h \ofraka, \hat\calG_{\delta,h}[\ofraka]\ou).
  \]
  We pass to the limit by using \eqref{eq:PhiAndProj} and \eqref{eq:Assumption2.16} to see that, as we needed,
  \[
    \limsup_{h \downarrow0} J(\ofraka_h,\ou_h) \leq J(\ofraka,\ou).
  \]

  \item We first observe that, using \eqref{eq:E2}, the estimate of Lemma~\ref{lem:DiscrStatesAreOKAbs}, and \eqref{eq:Assumption2.16} we have
  \begin{align*}
    \left| \frakB_\delta[\ofraka_h](\ou_h, \ou - \hat\calG_{\delta,h}[\ofraka]\ou ) \right| &\leq A \| \ou - \hat\calG_{\delta,h}[\ofraka]\ou \|_{X_\delta} \| \ou_h \|_{X_\delta} \\
    &\leq \frac{MA}\alpha \| f \|_{X_1} \| \ou - \hat\calG_{\delta,h}[\ofraka]\ou \|_{X_\delta} \to 0
  \end{align*}
  as $h \downarrow0$. With this at hand, and using \eqref{eq:E3}, we estimate
  \begin{align*}
    \alpha \| \ou - \ou_h \|_{X_\delta}^2 &\leq \frakB_\delta[\ofraka_h](\ou - \ou_h,\ou - \ou_h )\\
    & = \langle f, \ou_h \rangle - 2 \frakB_\delta[\ofraka_h](\ou,\ou_h) + \frakB_\delta[\ofraka_h](\ou,\ou) \\
      &=\langle f, \ou_h \rangle - 2 \frakB_\delta[\ofraka_h](\ou - \hat\calG_{\delta,h}[\ofraka]\ou,\ou_h)+ \frakB_\delta[\ofraka_h](\ou,\ou) \\ &-2\frakB_\delta[\ofraka_h](\ou_h,\hat\calG_{\delta,h}[\ofraka]\ou ) \\
      &= \langle f, \ou_h -2\hat\calG_{\delta,h}[\ofraka]\ou \rangle - 2 \frakB_\delta[\ofraka_h](\ou - \hat\calG_{\delta,h}[\ofraka]\ou,\ou_h)+ \frakB_\delta[\ofraka_h](\ou,\ou),
  \end{align*}
  where we used that $(\ofraka_h,\ou_h) \in \calZ_h^\delta$, and that $\hat\calG_{\delta,h}[\ofraka]\ou \in X_{\delta,h}$. Using \eqref{eq:ObjectiveAbs}, which defines our cost, we rewrite
  \begin{align*}
    \alpha \| \ou - \ou_h \|_{X_\delta}^2 +\phi(\ofraka_h) -\phi(\ofraka) &\leq 
      J(\ofraka_h,\ou_h) -J(\ofraka,\ou) + \langle f, \ou -2\hat\calG_{\delta,h}[\ofraka]\ou \rangle\\
      & \,\,\,- 2 \frakB_\delta[\ofraka_h](\ou - \hat\calG_{\delta,h}[\ofraka]\ou,\ou_h)+ \frakB_\delta[\ofraka_h](\ou,\ou).
  \end{align*}
  We can now pass to the limit, use the previous step, \eqref{eq:Assumption2.16}, our first observation, and \eqref{eq:E1} to see that, since $(\ofraka,\ou) \in \calZ^\delta$,
  \[
    \limsup_{h \downarrow0} \left( \alpha \| \ou - \ou_h \|_{X_\delta}^2 +\phi(\ofraka_h) -\phi(\ofraka) \right) \leq - \langle f, \ou \rangle + \frakB_\delta[\ofraka](\ou,\ou) = 0.
  \]
  Arguing along, but not relabeling, the sub-sequence where the limit superior is attained, we have that
  \[
    \alpha \lim_{h\downarrow0} \| \ou - \ou_h \|_{X_\delta}^2 \leq \lim_{h\downarrow0}\left( \alpha \| \ou - \ou_h \|_{X_\delta}^2 +\phi(\ofraka_h) -\phi(\ofraka) \right) =0,
  \]
  because the lower semicontinuity of $\phi$ and the first point imply that $\phi(\ofraka) \leq \lim_{h\downarrow0}\phi(\ofraka_h)$.

  \item From the previous step we see that
  $
    \lim_{h\downarrow0}\phi(\ofraka_h) \leq \phi(\ofraka),
  $
  which combined with lower semicontinuity gives the assertion.
\end{enumerate}
All assertions have been proved, and all convergences shown. The theorem is now proved.
\end{proof}

\subsubsection{Variational convergence as $\delta\downarrow0$}
\label{subsub:GammaConvergenceDiscrAbs}

The passage to the limit $\delta \downarrow0$ poses no difficulty.

\begin{theorem}[$\Gamma$--convergence for $h \in (0,h_0)$]
\label{thm:GconvAsDeltaTo0Discr}
Fix $h \in (0,h_0)$. The family of discrete reduced cost functionals $\{r^\delta_h\}_{\delta \in (0,1)}$ $\Gamma$-converge to $r_h^0$. Moreover, this family is equicoercive.
\end{theorem}
\begin{proof}
This repeats, without any substantial change, the proof of Theorem~\ref{thm:GconvAsDeltaTo0}.
\end{proof}

Once again, the previous result implies convergence of discrete minimizers.

\begin{corollary}[convergence of discrete minimizers]
\label{col:LimDelta0h>0}
Fix $h \in (0,h_0)$. Given a family of optimal design coefficients, \ie $\{\ofraka_{\delta,h} \}_{\delta \in (0,1)} \subset \calH_h$ satisfying
\[
  \ofraka_{\delta,h} \in \argmin_{\fraka_h \in \calH_h} r_h^\delta(\fraka_h),
\]
there is $\ofraka_h \in \argmin_{\fraka_h \in \calH_h} r_h^0(\fraka_h)$ such that, up to sub-sequences, $\ofraka_{\delta,h} \overset{d}{\to} \ofraka_h$, as $\delta \downarrow0$, and that 
\[
  \lim_{\delta\downarrow0} r_h^\delta(\ofraka_{\delta,h}) = r_h^0(\ofraka_h).
\]
Moreover, setting $\ou_{\delta,h} = T_h^\delta(\ofraka_{\delta,h})$ and $\ou_h = T_h^0(\ofraka_h)$ we also obtain that
\begin{equation}
  \lim_{\delta \downarrow0} \| \ou_h - \ou_{\delta,h} \|_{X_\delta} = 0.
\end{equation}
In addition, as $\delta \downarrow0$, $\ou_{\delta,h} \to \ou_h$ in $X_1$ and  $\phi(\ofraka_{\delta,h}) \to \phi(\ofraka_h)$.
\end{corollary}
\begin{proof}
The proof repeats that of Corollary~\ref{col:LimDelta0}.
\end{proof}

\subsection{Asymptotic compatibility}
\label{sub:ACAbs}

\begin{figure}
\label{fig:AC}
  \begin{center}
    \begin{tikzcd}
        ({ \ofraka_{\delta, h} }, { \ou_{\delta,h} })
          \arrow[d, harpoon, "h \downarrow 0"]
          \arrow[d, shift right=1.5]
          \arrow[dr, "k \uparrow \infty", harpoon]
          \arrow[dr, shift right=1.5]
          \arrow[r, "\delta \downarrow 0", harpoon, shift left=1.5]
          \arrow[r]
      &
        ({ \ofraka_h }, { \ou_h })
          \arrow[d, harpoon, "h \downarrow 0"]
          \arrow[d, shift right=1.5]
      \\
        ({ \ofraka_\delta }, { \ou_\delta  })
          \arrow[r, harpoon, "\delta \downarrow 0"]
          \arrow[r, shift right=1.5]
      &
        ({ \ofraka }, { \ou })
    \end{tikzcd}
  \end{center}
  \caption{Commutative diagram depicting the notion of asymptotic compatibility for optimal design problems; see Definition~\ref{def:ACAbs}. The symbol $\overset{\rightharpoonup}{\to}$ denotes that, along the corresponding sequence $\tau$, we have convergence of $\phi(\ofraka_\tau)$ and $\ou_\tau$ in $X_1$.}
\end{figure}
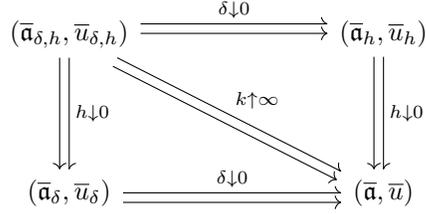

We now are finally ready to study the asymptotic compatibility of our class of optimal design problems. We begin with its definition, which is graphically illustrated in Figure~\ref{fig:AC}.

\begin{definition}[asymptotic compatibility]
\label{def:ACAbs}
Let $\{(\ofraka_{\delta,h},\ou_{\delta,h})\}_{\delta \in (0,1),h \in (0,h_0)}$ be a family of solutions to \eqref{eq:DiscrOptDesignDeltaAbs}. We say that this family is asymptotically compatible with respect to discretization and modeling parameters if, for any (not relabeled) sub-sequences that satisfy $h \downarrow 0$ and $\delta \downarrow0$, there is a further sub-sequence $\{(\delta_k,h_k)\}_{k=1}^\infty$ with $(\delta_k,h_k) \to (0,0)$, and a pair $(\ofraka,\ou) \in \calZ^0$ such that $\phi(\ofraka_{\delta_k,h_k}) \to \phi(\ofraka)$, $\ou_{\delta_k,h_k} \to \ou$ in $X_1$, and $(\ofraka,\ou)$ solves \eqref{eq:OptDesignDeltaAbs} for $\delta = 0$.
\end{definition}

The previous definition deserves some comment. The notion of asymptotic compatibility for a class of well-posed, parametrized linear problems was introduced and developed in \cite{MR3231986,MR4064535}. It, essentially, asserted the convergence of discretizations along all discretization and parameter sub-sequences, and that the limiting object solves the corresponding problem. It is also important to note that, since the problems are assumed well-posed, \emph{the limiting object is unique}. In our setting, on the other hand, we cannot guarantee uniqueness of solutions, regardless of the value of the discretization or modeling parameters. For this reason, a different limiting solution may be reached during the sub-sequence selection process.

We finally comment that, depending on the chosen discretization scheme, certain problems can also exhibit \emph{conditional compatibility}, where convergence to a solution is only possible for a particular sub-sequence of parameters, The interested reader is referred to \cite{MR3231986} for an example in the linear setting. The possibility of constructing conditionally compatible schemes for the problems we discuss here is under investigation.

Before continuing, we make one final set of assumptions. The first one essentially strengthens \eqref{eq:C1} and it is copied from \cite[Assumption 5v]{du2024asymptoticallycompatibleschemesnonlinear}; see also \cite[Assumption 4ii]{MR3231986}. The second one, on  the other hand, states that the discretization scheme for the state equation itself is asymptotically compatible. It should be compared with \eqref{eq:Assumption2.16}, where only the discretization parameter is changing. It may be viewed as a variant of the main result in \cite{du2024asymptoticallycompatibleschemesnonlinear} but where, additionally, the coefficients are allowed to change in a specific way.

\begin{itemize}
  \item \textbf{Asymptotic approximation property:} For any $v \in X_0$ and sequences of parameters $\{(\delta_k,h_k)\}_{k=1}^\infty$ satsfying
  \[
    \begin{dcases}
      \lim_{k \uparrow \infty} \delta_k = 0, \ h_k = h, & h \in (0,h_0), \,\,\text{ or} \\
      \lim_{k \uparrow \infty} \delta_k = \lim_{k \uparrow \infty} h_k = 0,
    \end{dcases}
  \]
  there is $\{v_k \in X_{\delta_k,h_k}\}_{k=1}^\infty$ such that $v_k \to v$ in $X_0$.

  \item \textbf{Asymptotic compatibility of state equations:} Let $(\fraka,u) \in \calZ^0$. For any sequence $\{(\delta_k,h_k)\}_{k=1}^\infty$ with $(\delta_k,h_k) \to (0,0)$ we have
  \begin{equation}
  \label{eq:StateIsAC}
    \left\| u - T_{h_k}^{\delta_k}(\Pi_{h_k}\fraka) \right\|_{X_1} \to 0, \qquad k \uparrow \infty.
  \end{equation}
\end{itemize}

With these assumptions at hand we prove our main abstract result: the asymptotic compatibility of our parametrized discrete optimal design problems.

\begin{theorem}[asymptotic compatibility]
\label{thm:ACAbs}
Let $\{(\ofraka_{\delta,h}, \ou_{\delta,h})\}_{\delta \in (0,1),h \in (0,h_0)}$ be a family of solutions to problem \eqref{eq:DiscrOptDesignDeltaAbs}. Under our running assumptions this family is asymptotically compatible in the sense of Definition~\ref{def:ACAbs}. Moreover, with the notation of this definition, we additionally have
\[
  \lim_{k \uparrow \infty} J(\ofraka_{\delta_k,h_k},\ou_{\delta_k,h_k}) = J(\ofraka,\ou),
\]
and
\begin{equation}
\label{eq:BBMTrickAbs}
  \lim_{k \uparrow\infty} \| \ou - \ou_{\delta_k,h_k} \|_{X_{\delta_k,h_k}} = 0.
\end{equation}
\end{theorem}
\begin{proof}
We begin by simplifying notation. Let $(\ofraka_k,\ou_k) = (\ofraka_{\delta_k,h_k},\ou_{\delta_k,h_k})$. By compactness of $\calH$ there is a (not relabeled) sub-sequence, and $\ofraka \in \calH$ such that $\ofraka_k \overset{d}{\to} \ofraka$. Set $\ou = T^0(\ofraka)$ so that $(\ofraka,\ou) \in \calZ^0$.

Let us now show the convergence of objective values. To achieve this we first recall that \eqref{eq:DiscrEmbedding} implies that $\calG_{0,h_k}[\ofraka]\ou \in X_{0,h_k} \hookrightarrow X_{\delta_k,h_k}$. Next, we need to show that
\[
  \lim_{k \uparrow \infty} \frakB_{\delta_k}[\ofraka_k](\calG_{0,h_k}[\ofraka]\ou,\calG_{0,h_k}[\ofraka]\ou) = \frakB_0[\ofraka](\ou,\ou).
\]
To see this we first note that
\begin{multline*}
  \left| \frakB_{\delta_k}[\ofraka_k](\ou,\ou) - \frakB_{\delta_k}[\ofraka_k](\calG_{0,h_k}[\ofraka]\ou,\calG_{0,h_k}[\ofraka]\ou) \right| \\
   = \left| \frakB_{\delta_k}[\ofraka_k](\ou - \calG_{0,h_k}[\ofraka]\ou, \ou + \calG_{0,h_k}[\ofraka]\ou ) \right|
  \leq A \| \ou - \calG_{0,h_k}[\ofraka]\ou \|_{X_{\delta_k}} \| \ou + \calG_{0,h_k}[\ofraka]\ou \|_{X_{\delta_k}} \\
  \lesssim \| \ou - \calG_{0,h_k}[\ofraka]\ou \|_{X_0} \| \ou + \calG_{0,h_k}[\ofraka]\ou \|_{X_0},
\end{multline*}
where the constant is independent of $\delta_k$ and $h_k$. Owing to \eqref{eq:C1} the first factor tends to zero whereas the second is bounded, thus the product tends to zero. Next, since $\ou \in X_0$, \eqref{eq:2.22} combined with \eqref{eq:2.23} gives
\[
  \lim_{k \uparrow \infty} \frakB_{\delta_k}[\ofraka_k](\ou,\ou) = \frakB_0[\ofraka](\ou,\ou),
\]
and this implies the claim. We then immediately infer that
\[
  \lim_{k \uparrow \infty} E_{\delta_k}[\ofraka_k](\calG_{0,h_k}[\ofraka]\ou) = E_0[\ofraka](\ou).
\]

Next, as we have done several times before, we use the characterizations of $\ou_k$ and $\ou$, respectively, in terms of the value of the energy to assert
\[
  \limsup_{k \uparrow \infty} \left( -\frac12 \frakB_{\delta_k}[\ofraka_k](\ou_k,\ou_k) \right) \leq \lim_{k \uparrow \infty } E_{\delta_k}[\ofraka_k](\calG_{0,h_k}[\ofraka]\ou) = E_0[\ofraka](\ou) = -\frac12 \frakB_0[\ofraka](\ou,\ou).
\]
Consequently,
\begin{equation}
\label{eq:2.91}
\langle f , \ou \rangle \leq \liminf_{k \uparrow \infty } \langle f, \ou_k \rangle.
\end{equation}
This result, and the lower semicontinuity of $\phi$ then imply, by arguments presented before, that
\[
  J(\ofraka,\ou) \leq \liminf_{k \uparrow \infty} J(\ofraka_k,\ou_k).
\]

We now prove the converse inequality. Recall that $(\Pi_{h_k}\ofraka, T^{\delta_k}_{h_k}(\Pi_{h_k}\ofraka)) \in \calZ_{h_k}^{\delta_k}$ and then
\[
  J(\ofraka_k,\ou_k) \leq J( \Pi_{h_k}\ofraka, T^{\delta_k}_{h_k}(\Pi_{h_k}\ofraka) ).
\]
We may now use \eqref{eq:PhiAndProj} and \eqref{eq:StateIsAC} to pass to the limit and get
\[
  \limsup_{k \uparrow \infty} J(\ofraka_k, \ou_k ) \leq \lim_{k\uparrow \infty} \langle f, T^{\delta_k}_{h_k}(\Pi_{h_k}\ofraka) \rangle + \lim_{k\uparrow \infty} \phi(\Pi_{h_k}\ofraka) = \langle f, \ou \rangle + \phi(\ofraka) = J(\ofraka,\ou),
\]
which, as we needed, shows convergence of the objective values.

Now, to show that $(\ofraka,\ou)$ solves \eqref{eq:OptDesignDeltaAbs} with $\delta =0$, let $(\frakb,v) \in \calZ^0$ and observe that
\[
  J(\ofraka_k,\ou_k) \leq J(\Pi_{h_k}\frakb, T^{\delta_k}_{h_k}(\Pi_{h_k}\frakb) ).
\]
Pass to the limit and use previously presented arguments to conclude that $J(\ofraka,\ou) \leq J(\frakb,v)$.

Observe that, owing to \eqref{eq:S1}, $\ou_k \to \ou$ in $X_1$ is implied by \eqref{eq:BBMTrickAbs}. This convergence, together with convergence of objective values, will in turn imply $\phi(\ofraka_k) \to \phi(\ofraka)$. It remains then to prove \eqref{eq:BBMTrickAbs}. We begin by observing that lower semicontinuity of $\phi$, together with convergence of objective values yields
\[
   \limsup_{k \uparrow \infty} \langle f, \ou_k \rangle = \lim_{k \uparrow \infty} J(\ofraka_k,\ou_k) - \liminf_{k \uparrow \infty}\phi(\ofraka_k) \leq J(\ofraka,\ou) - \phi(\ofraka) = \langle f, \ou \rangle,
\]
which combined with \eqref{eq:2.91} yields $\langle f, \ou_k \rangle \to \langle f, \ou \rangle$. Next, consider
\begin{align*}
  \alpha \| \ou - \ou_k \|_{X_{\delta_k}}^2 &\leq \frakB_{\delta_k}[\ofraka_k](\ou - \ou_k,\ou - \ou_k) = \langle f, \ou_k \rangle - 2 \frakB_{\delta_k}[\ofraka_k](\ou,\ou_k)  + \frakB_{\delta_k}[\ofraka_k](\ou,\ou) \\
  &= \langle f, \ou_k \rangle -2 \frakB_{\delta_k}[\ofraka_k](\calG_{0,h_k}[\ofraka]\ou,\ou_k) -2 \frakB_{\delta_k}[\ofraka_k](\ou - \calG_{0,h_k}[\ofraka]\ou,\ou_k) \\
  &+ \frakB_{\delta_k}[\ofraka_k](\ou,\ou) \\
  &= \langle f, \ou_k \rangle -2\langle f, \calG_{0,h_k}[\ofraka]\ou \rangle-2 \frakB_{\delta_k}[\ofraka_k](\ou - \calG_{0,h_k}[\ofraka]\ou,\ou_k) \\
  &+ \frakB_{\delta_k}[\ofraka_k](\ou,\ou).
\end{align*}
Let us now pass to the limit. We just showed how to deal with the first term. The second one converges to $-2\langle f, \ou \rangle$ owing to \eqref{eq:Assumption2.16}. The third one can be controlled as
\[
  \left| \frakB_{\delta_k}[\ofraka_k](\ou - \calG_{0,h_k}\ou,\ou_k) \right| \leq A \| \ou - \calG_{0,h_k}\ou \|_{X_{\delta_k}} \| \ou_k \|_{X_{\delta_k}} \lesssim \| \ou - \calG_{0,h_k}\ou \|_{X_0}
\]
with a constant independent of $\delta_k$ and $h_k$. Here we used \eqref{eq:E2}, the estimate of Lemma~\ref{lem:DiscrStatesAreOKAbs}, and the uniform embedding into $X_0$. We can then conclude, using again \eqref{eq:Assumption2.16}, that this term vanishes in the limit. Finally, since $\ou \in X_0$, we use \eqref{eq:2.22} and \eqref{eq:2.23} to get
\[
  \frakB_{\delta_k}[\ofraka_k](\ou,\ou) \to \frakB_0[\ofraka](\ou,\ou) = \langle f, \ou \rangle.
\]
Combining all the terms \eqref{eq:BBMTrickAbs} is shown.
\end{proof}

\section{Optimal design in nonlocal conductivity}
\label{sec:Conductivity}

We now present the first application of the abstract framework developed in Section~\ref{sec:AbstractSetting}: optimal design in a nonlocal conductivity problem, its analysis, discretization, and asymptotic compatibility.

\subsection{Notation}
\label{sub:Notation}

Before we start applying the abstract framework developed in the previous section, let us introduce  some of the notation we will use in the remaining sections. 
From now on we will assume that $\Omega \subset \R^n$, with $n \in \N$, is a bounded domain with Lipschitz boundary. The parameter $R>0$ shall be called the \emph{horizon} and we set $\Omega_R = \Omega + B_R$, where $B_R$ is the ball of radius $R$ centered at the origin. The \emph{nonlocal boundary} is $\Omega_R \setminus \Omega$. Finally, we set $\calD_R = \left( \Omega \times \Omega_R \right) \cup \left( \Omega_R \times \Omega \right)$.

We will adhere to standard notation with regard to spaces of integrable functions. In addition, we will use Sobolev spaces of integer order $W^{k,p}(D)$, where $D \subset \R^n$ is a domain with Lipschitz boundary, $k \in \N$, and $p \in (1,\infty)$. The closure of $C_0^\infty(D)$ in $W^{k,p}(D)$ is $W^{k,p}_0(D)$. In addition, $H^k(D) = W^{k,2}(D)$ and $H^k_0(D) = W^{k,2}_0(D)$. We will also need to use Sobolev spaces of fractional order. For $s \in (0,1)$ and $p \in (1,\infty)$ we define
\begin{align*}
  W^{s,p}(D) &= \left\{ w \in L^p(D) \ \middle| \ |w|_{W^{s,p}(D)} < \infty \right\}, \\
  |w|_{W^{s,p}(D)}^p &= \gamma_{s,n,p} \iint_{D \times D} \left|\frac{ w(x) - w(y) }{|x-y|^s}\right|^p \frac{\diff x \diff y}{|x-y|^n}, \\
  \| w \|_{W^{s,p}(D)}^p &= \| w \|_{L^p(D)}^p + |w|_{W^{s,p}(D)}^p,
\end{align*}
where
\[
  \gamma_{s,n,p} = p(1-s) \left( \int_{\polS^{n-1}} \omega_1^p \diff \sigma(\omega) \right)^{-1}.
\]
The normalization constant $\gamma_{s,n,p}$ that appears on the seminorm is chosen to guarantee that (see \cite{MR3586796})
\[
  |w|_{W^{s,p}(D)}^p \to \| \GRAD w \|_{\bL^p(D)}^p, \qquad s \uparrow 1,
\]
whenever this makes sense. Moreover,
\[
  \widetilde{W}^{s,p}(D) = \left\{ w \in W^{s,p}(\R^n) \ \middle| \ \supp(w) \subset D \right\}.
\]
For $p=2$ we set $H^s(D) = W^{s,2}(D)$ and $\widetilde{H}^s(D)  = \widetilde{W}^{s,2}(D)$. Finally, we mention that spaces of vector-valued functions and their elements will be written using bold typeface.

Let $k \in \N_0$. By $\polP_k$ we denote the space of polynomials of total degree at most $k$. To deal with discretizations we assume that $\Omega$ is a polytope, so that it can be triangulated exactly. By $\Tr = \{\Triang_h\}_{h \in (0,h_0)}$ we denote a quasi-uniform, in the sense of \cite{MR4242224,MR2373954,MR520174}, family of conforming triangulations of $\Omega$ of size $h>0$. Here $h_0$ is, say, the mesh size of the coarsest triangulation.

For $r,k \in \N_0$ we define
\[
  \calL_k^r(\Triang_h) = \left\{ w_h \in C^{r-1}(\bar\Omega) \ \middle| \ w_{h|T} \in \polP_k, \ \forall T \in \Triang_h \right\},
\]
with the understanding that functions in $C^{-1}(\bar\Omega)$ are simply bounded in $\Omega$. We define $V_h = \calL_1^1(\Triang_h) \cap W^{1,1}_0(\Omega)$ and observe that such functions can be trivially extended by zero to $\Omega_R$. When no confusion arises, we will make no notational distinction between a function and its zero extension.

\subsection{Setup}
\label{sub:SetupConductivity}

Let us now describe our problem, and verify that it fits the framework developed in Section~\ref{sec:AbstractSetting}. We begin with a slight change of notation. Namely, we shall replace the parameter $\delta$, which we think of as \emph{degree of nonlocality}, by $s \in (0,1)$ (differentiability order), via the formal replacement $\delta = 1-s$. Thus the limit we care about is $s \uparrow 1$. In what follows we shall index everything using $s$ and not $\delta$.

\begin{itemize}
  \item Let $0<a_{\min} \leq a_{\max}$ be constants. The design space is the set of functions
  \[
    \calH = \left\{ a \in L^\infty(\Omega_R) \ \middle| \ a_{\min} \leq a(x) \leq a_{\max} \mae \Omega \right\}
  \]
  under the weak-* convergence in $L^{\infty}$. Since $\calH$ is a bounded subset of the Banach space $L^\infty(\Omega_R)$, the Banach-Alaoglu theorem \cite[Theorem 3.16]{MR2759829} guarantees that sequences in $\calH$ always possess subsequences that converge weak-* to an element in $L^{\infty}(\Omega_R)$. It is also straightforward to show that the limit  is bounded from below and above by $a_{\min}$ and $a_{\max}$, respectively, demonstrating  the compactness of $\calH$ under weak-* convergence.  Moreover, since $L^1(\Omega)$ is separable and $(L^1(\Omega))' = L^\infty(\Omega)$, the set $\calH$ is metrizable in the weak-* topology of $L^{\infty}(\Omega)$,  \cite[Theorem 3.28]{MR2759829}. This is the metric that shall be denoted by $d$.

  \item Set $X_1 = H^1_0(\Omega)$, $X_s= \widetilde{H}^{s}(\Omega)$ for $s \in (0,1)$, $X_0 = L^2(\Omega)$. Owing to the choice of normalization constant in the norm of $\widetilde{H}^s(\Omega)$, this scale satisfies a uniform embedding property; see \cite[Remark 6]{MR3586796}.

  \item We now define the bilinear forms that define the state equations:
  \begin{equation}
  \label{eq:BLFormsConductivity}
    \begin{aligned}
      \frakB_1[\fraka](v,w) &= \int_\Omega \fraka(x) \GRAD w(x)\cdot \GRAD v(x) \diff x, \\
      \frakB_{s}[\fraka](v,w) &= \gamma_{s,n,2}\iint_{\calD_R} \frakA(x,y) \frac{ v(x) - v(y) }{|x-y|^s} \frac{ w(x) - w(y) }{|x-y|^s} \frac{\diff x \diff y}{|x-y|^n},
    \end{aligned}
  \end{equation}
  where $s \in (0,1)$. Here and in what follows, for $\fraka \in \calH$ and $x,y \in \Omega_R$, we set $\frakA(x,y) = \tfrac12 \left( \fraka(x) + \fraka(y) \right)$. 
  \item Let $q \in (1,\infty)$ and $\Lambda \in L^\infty(\Omega)$ be strictly positive, \ie there is $\lambda > 0$ for which we have $\Lambda(x) \geq \lambda$ \mae $\Omega$. Finally, the function $\phi$ shall be
  \[
    \phi(\fraka) = \int_\Omega \Lambda(x) |\fraka(x)|^q \diff x.
  \]
  Indeed, since $\Omega$ is bounded, $L^{\tfrac{q}{q-1}}(\Omega) \subset L^1(\Omega)$. Thus, weak-* convergence in $L^\infty(\Omega)$ implies weak convergence in $L^q(\Omega)$ which, in turn, yields lower semicontinuity of $\phi$.
\end{itemize}

Let us then, for clarity, state the class of problems we are currently interested in. We let $s \in (0,1]$, $q \in (1,\infty)$ and assume that $f \in L^2(\Omega)$ is given. We seek to find a pair $(\fraka,u) \in \calH \times \widetilde{H}^s(\Omega)$ that minimizes
\begin{equation}
\label{eq:ObjConductivity}
  J(\fraka,u) = \int_\Omega f(x) u(x) \diff x + \int_\Omega \Lambda(x) |\fraka(x) |^q \diff x,
\end{equation}
subject to
\begin{equation}
\label{eq:StateConductivity}
  \frakB_{s}[\fraka](u,v) = \int_\Omega f(x) v(x) \diff x, \qquad \forall v \in \widetilde{H}^s(\Omega).
\end{equation}
Here, for uniformity, we have set $\widetilde{H}^1(\Omega) = H^1_0(\Omega)$.

\subsection{Verification of structural assumptions and results}
\label{sub:AnalysisConductivity}

Let us now verify that our optimal design in nonlocal conductivity verifies all the assumptions of our abstract framework.

\begin{itemize}
  \item \textbf{Uniform embedding:} This holds owing to the fractional Poincar\'e inequality with uniform constant; see \cite[Lemma 3.2]{MR4629861}, \cite[Theorem 1.1]{MR2041005}, and the references therein for a proof.

  \item \textbf{Asymptotic compactness:} This is \cite[Theorem 1.2]{MR2041005}.

  \item \textbf{Strong continuity:} Let $s=1$. If $v \in H^1_0(\Omega)$, then $|\GRAD v|^2 \in L^1(\Omega)$. Since $\fraka_j \rightharpoonup^* \fraka$ in $L^\infty(\Omega)$,
  \[
    \int_\Omega \fraka_j(x) |\GRAD v(x)|^2 \diff x \to \int_\Omega \fraka(x) |\GRAD v(x)|^2 \diff x.
  \]
  For the case $s \in (0,1)$ we observe that if $v \in \widetilde{H}^s(\Omega)$, then the mapping $\calD_R \ni (x,y) \mapsto \tfrac{|v(x) - v(y)|^2}{|x-y|^{n+2s}} \in \R$ belongs to $L^1(\calD_R)$. Arguing as in the previous case, we then see that a sufficient condition is that the coefficient $\frakA_j \rightharpoonup^* \frakA$ in $L^\infty(\calD_R)$. This is the content of the following result.
\end{itemize}

\begin{lemma}[weak-* convergence]
\label{lem:aimpliesA}
Let $\{\fraka_j\}_{j=1}^\infty \subset \calH$ and $\fraka \in \calH$ be such that $\fraka_j \rightharpoonup^* \fraka$ in $L^\infty(\Omega_R)$. Define
\[
  \frakA_j(x,y) = \frac12 \left( \fraka_j(x) + \fraka_j(y) \right), \qquad \frakA(x,y) = \frac12 \left( \fraka(x) + \fraka(y) \right).
\]
Then we have $\frakA_j \rightharpoonup^* \frakA$ in $L^\infty(\calD_R)$.
\end{lemma}
\begin{proof}
Let $\psi \in L^1(\calD_R)$, then
\[
  \iint_{\calD_R} \frakA_j(x,y) \psi(x,y) \diff x \diff y = \frac12 \left[ \iint_{\calD_R} \fraka_j(x) \psi (x,y) \diff x \diff y + \iint_{\calD_R} \fraka_j(y) \psi (x,y) \diff x \diff y \right].
\]
In addition, since $\fraka_j \in \calH$, we may estimate
\begin{align*}
  \iint_{\calD_R}|\fraka(x) \psi(x,y)| \diff x \diff y &\leq a_{\max} \iint_{\calD_R} |\psi(x,y)| \diff x \diff y, \\
  \iint_{\calD_R}|\fraka(y) \psi(x,y)| \diff x \diff y &\leq a_{\max} \iint_{\calD_R} |\psi(x,y)| \diff x \diff y.
\end{align*}
This justifies the application of Fubini's Theorem to obtain, for instance,
\[
  \iint_{\calD_R} \fraka_j(x) \psi(x,y) \diff x \diff y = \int_\Omega \fraka_j(x) \int_{\Omega_R} \psi(x,y) \diff y \diff x+ \int_{\Omega} \fraka_j(y) \int_{\Omega_R} \psi(x,y) \diff x \diff y.
\]
Next we observe that, since $\psi \in L^1(\calD_R)$, for almost every $x \in \Omega_R$ the mapping $\Omega_R \ni y \mapsto v(x,y)$ belongs to $L^1(\Omega_R)$. Thus we may pass to the limit in the previous identity. Collecting back and applying, again, Fubini's Theorem we get the result.
\end{proof}

\begin{itemize}
  \item \textbf{Uniform boundedness:} For $s=1$ we have
  \[
    \left| \int_\Omega \fraka(x) |\GRAD v(x)|^2 \diff x \right| \leq a_{\max} \| \GRAD v \|_{\bL^2(\Omega)}^2.
  \]
  Similarly, for $s \in (0,1)$,
  \[
    \left| \gamma_{s,n,2} \iint_{\calD_R} \frakA(x,y) \frac{|v(x)-v(y)|^2}{|x-y|^{2s}} \frac{\diff x \diff y}{|x-y|^n} \right| \leq a_{\max} |v|_{\widetilde{H}^s(\Omega)}^2
  \]
  In summary, $A = a_{\max}$.

  \item \textbf{Uniform coercivity:} We repeat the previous arguments to obtain that $\alpha = a_{\min}$.

  \item \textbf{Parametric continuity:} A result similar to \eqref{eq:2.22} was proven in \cite[Theorem 8]{MR3339075}. The only difference is that, there, the result is proven for the case $v = w$ and when the integration is over $\Omega_R \times \Omega_R$ instead of $\calD_R$. However, the same proof applies to our case as well. In \cite[Theorem 5.4]{MR4629861} a proof of \eqref{eq:2.23} is given for the elasticity and peridynamics cases. Once again, the proof can be adapted to the case we are interesed here.

  \item \textbf{Lower semicontinuity:} See \cite[Theorem 1]{MR4201703} for a proof.
\end{itemize}

Having verified the structural assumptions, we expediently obtain existence of minimizers for problem \eqref{eq:ObjConductivity}--\eqref{eq:StateConductivity}.

\begin{lemma}[uniform well posedness]
For every $s \in (0,1]$ and each coefficient $\fraka \in \calH$ there is a unique $u_s \in \widetilde{H}^s(\Omega)$ that solves \eqref{eq:StateConductivity}. In addition, this solution satisfies
\[
  | u_s |_{\widetilde{H}^s(\Omega)} \lesssim \| f \|_{L^2(\Omega)},
\]
where the implicit constant is independent of $s$.
\end{lemma}
\begin{proof}
This is Lemma~\ref{lem:StatesAreOKAbs} in the conductivity setting.
\end{proof}

Owing to this we can define the design-to-state mapping $T^{s} : \calH \to \widetilde{H}^s(\Omega)$ and the reduced costs $\{r^{s}\}_{s \in (0, 1]}$.

\begin{theorem}[existence]
For every $s \in (0,1]$, problem \eqref{eq:ObjConductivity}--\eqref{eq:StateConductivity} has a solution.
\end{theorem}
\begin{proof}
Invoke Theorem~\ref{thm:ExistenceAbs}.
\end{proof}

Having shown existence of minima, we can show the variational convergence of these problems.

\begin{theorem}[convergence of minimizers]
Let $\{\ofraka_s\}_{s \in (0,1)}$ be a family of minimizers of the reduced costs $\{r^{s}\}_{s \in (0,1)}$. In other words, the family of pairs $\{(\ofraka_s,\ou_s=T^{s}(\ofraka_s))\}_{s \in (0,1))}$ solve problem \eqref{eq:ObjConductivity}--\eqref{eq:StateConductivity}. Then, as $s \uparrow 1$,
\begin{itemize}
  \item The family $\{r^{s}\}_{s \in (0,1)}$ $\Gamma$-converges to $r^1$ in the metric of weak-* convergence in $\calH$. Moreover, this family is equicoercive.

  \item There exists $\ofraka_1 \in \calH$ so that, up to sub-sequences, $\ofraka_s \rightharpoonup^* \ofraka_1$ in $L^\infty(\Omega)$. Moreover,
  \[
    \lim_{s \uparrow 1} r^{s}(\ofraka_s) = r^1(\ofraka_1), \qquad \lim_{s \uparrow 1} \| \ou_1 - \ou_s \|_{\widetilde{H}^s(\Omega)} = 0,
  \]
  we have $\ou_s \to \ou_1$ in $L^2(\Omega)$. Finally $\ofraka_s \to \ofraka_1$ in $L^q(\Omega)$ and $(\ofraka_1,\ou_1)$ is a solution to \eqref{eq:ObjConductivity}--\eqref{eq:StateConductivity} when $s = 1$.
\end{itemize}
\end{theorem}
\begin{proof}
This result combines Theorem~\ref{thm:GconvAsDeltaTo0} and Corollary~\ref{col:LimDelta0}.
\end{proof}

\subsection{Discretization and results}
\label{sub:DiscretizationConductivity}

Here we specify the discretization scheme we shall employ to treat problem \eqref{eq:ObjConductivity}--\eqref{eq:StateConductivity}. We recall that we are assuming that $\Omega$ is a polytope. For every $s \in (0,1]$ we set $X_{s,h} = V_h$ algebraically, but with the norm of $\widetilde{H}^s(\Omega)$. In addition, we set $\calH_h = \calL^{0}_0(\Triang_h) \cap \calH$, \ie the set of piecewise constant, with respect to $\Triang_h$, functions that belong to $\calH$. Clearly, this set is weak-* closed in $L^\infty(\Omega)$.

Let us now verify the assumptions on the discretization.

\begin{itemize}
  \item \textbf{Discrete embedding:} Since our discretization is conforming, this follows from the uniform embedding of Section~\ref{sub:AnalysisConductivity}.

  \item \textbf{Approximation property:} This can be achieved by using the Scott-Zhang or Chen-Nochetto interpolants and their properties; see \cite{MR1011446,MR1742264,MR3702417} for $s=1$, and \cite{MR3118443,MR4238777,MR4053241} for $s \in (0,1)$.

  \item \textbf{Density:} We let $\Pi_h$ be the $L^2(\Omega)$ projection onto $\calL^{-1}_0(\Triang_h)$, \ie
  \[
    \Pi_h \fraka = \sum_{T \in \Triang_h} \left(\fint_T \fraka(x) \diff x \right) \chi_T,
  \]
  where $\fint_E w(x) \diff x = \tfrac1{|E|}\int_E w(x) \diff x$, and $\chi_S$ denotes the characteristic function of the set $S$. Clearly, if $\fraka\in\calH$, then $\Pi_h \fraka \in \calH_h$. Moreover, by properties of projections, we have that $\Pi_h \fraka \to \fraka$ in $L^2(\Omega)$ as $h \downarrow0$. Since $\Omega$ is bounded this implies convergence in any $L^r(\Omega)$ for $r \in [1,2]$. On the other hand, if $r \in (2,\infty)$, by interpolation we have
  \begin{equation}
  \label{eq:CoeffLrConvConductivity}
    \begin{aligned}
      \| \Pi_h \fraka - \fraka \|_{L^r(\Omega)} &\leq \| \Pi_h \fraka - \fraka \|_{L^2(\Omega)}^{2/r} \| \Pi_h \fraka - \fraka \|_{L^\infty(\Omega)}^{1-2/r} \\
         &\leq (2a_{\max})^{1-2/r} \| \Pi_h \fraka - \fraka \|_{L^2(\Omega)}^{2/r} \to 0,
    \end{aligned}
  \end{equation}
  as $h \downarrow0$. This then shows that, for every $\fraka \in \calH$, we have $\phi(\Pi_h\fraka) \to \phi(\fraka)$ as $h \downarrow0$.

  For future reference, we also record the fact that, up to sub-sequences, $\Pi_h \fraka \rightharpoonup^* \fraka$ in $L^\infty(\Omega)$ as $h \downarrow0$. Indeed, the convergence in $L^2(\Omega)$ implies that there is a (not relabeled) subequence such that $\Pi_h \fraka \to \fraka$ almost everywhere. In addition, if $\phi \in L^1(\Omega)$ is arbitrary, then
  \[
    \left| (\Pi_h \fraka - \fraka ) \phi \right| \leq 2a_{\max}|\phi| \in L^1(\Omega).
  \]
  By dominated convergence, we then conclude that
  \[
    \int_\Omega \left(\Pi_h \fraka(x) - \fraka(x) \right) \phi(x) \diff x \to 0.
  \]

  \item \textbf{Perturbed Galerkin projection:} Let $\fraka \in \calH$, $s \in (0,1]$, and $v = T^s[\fraka] \in \widetilde{H}^s(\Omega)$. Recall that, in our setting, given $v \in \widetilde{H}^s(\Omega)$ the functions $\calG_{s,h}[\fraka]v,\hat\calG_{s,h}[\fraka]v \in V_h$ are such that
  \[
    \frakB_s[\fraka](v,w_h) = \frakB_s[\fraka](\calG_{s,h}[\fraka]v,w_h) = \frakB_s[\Pi_h\fraka](\hat\calG_{s,h}[\fraka]v,w_h), \qquad \forall w_h \in V_h.
  \]
  Moreover, we have that $\calG_{s,h}[\fraka]v \to v$ in $\widetilde{H}^s(\Omega)$, as $h \downarrow0$. Thus, upon setting $e = v - \calG_{s,h}[\fraka]v$ and $\hat e = v - \hat\calG_{s,h}[\fraka]v$, we can estimate
  \begin{equation}
  \label{Eq: mainPerturbed}
    \begin{aligned}
      a_{\min} | \hat e |_{\widetilde{H}^s(\Omega)}^2 &\leq \frakB_s[\Pi_h\fraka](\hat e,\hat e)  \\
        &= \frakB_s[\Pi_h\fraka](\hat e, e) + \frakB_s[\Pi_h \fraka]\left(\hat e, \left( \calG_{s,h}-\hat\calG_{s,h} \right)[\fraka]v \right)  \\
      &= \mathrm{I} + \mathrm{II}.
    \end{aligned}
  \end{equation}
  For the first term we obtain
  \[
    \mathrm{I} \leq a_{\max} | e |_{\widetilde{H}^s(\Omega)}  | \hat  e |_{\widetilde{H}^s(\Omega)}
    \leq \frac{a_{\min} }4 | \hat  e |_{\widetilde{H}^s(\Omega)}^2 + \frac{a_{\max}^2}{a_{\min}} | e |_{\widetilde{H}^s(\Omega)}^2.
  \]
  For the second one, we use the definition of the perturbed Galerkin projection to obtain
  \begin{align*}
    \mathrm{II} &= \frakB_s[\Pi_h \fraka - \fraka]\left(v, ( \calG_{s,h}[\fraka]-\hat\calG_{s,h}[\fraka] )v \right) = \frakB_s[\fraka - \Pi_h\fraka](v,e) + \frakB_s[\Pi_h \fraka - \fraka](v,\hat e) \\
    &\leq 2A |v|_{\widetilde{H}^s(\Omega)} |e|_{\widetilde{H}^s(\Omega)} + \frakB_s[\Pi_h \fraka - \fraka](v,\hat e),
  \end{align*}
  where we also used that, in our particular setup and for every $w_1,w_2 \in \widetilde{H}^s(\Omega)$, the mapping $\frakb \mapsto \frakB_s[\frakb](w_1,w_2)$ is linear. Gathering all the obtained estimates we see that
  \[
    | \hat e |_{\widetilde{H}^s(\Omega)}^2 \lesssim |e|_{\widetilde{H}^s(\Omega)}^2 + |e|_{\widetilde{H}^s(\Omega)} + \frakB_s[\Pi_h \fraka - \fraka](v,\hat e).
  \]
  The first two terms tend to zero as $h \downarrow0$. The last term, however, requires special treatment. Boundedness of $\fraka$ alone is not enough to either absorb this term on the left hand side, or to show that it vanishes in the limit. For this reason, we examine it separately for each case $s=1$ and $s \in (0,1)$.
  \begin{itemize}
    \item \textbf{The case $s=1$:} Let $(\fraka,v) \in \calZ^1$. We recall that, since $f \in L^2(\Omega)$, owing to \cite[Theorem 2]{MR159110}, see also \cite[Section 2.2]{MR3129757}, there is $\epsilon>0$ such that, assuming that $\fraka \in \calH$ only, the solution to \eqref{eq:StateConductivity} satisfies
    \[
      \|\GRAD v\|_{\bL^{2+\epsilon}(\Omega)} \leq C \| f \|_{L^2(\Omega)},
    \]
    where $\epsilon$ and $C$ depend only on $d$ and the ratio $\tfrac{a_{\max}}{a_{\min}}$. With this at hand we let now $r \in (1,\infty)$ be such that $\tfrac1r + \tfrac1{2+\epsilon} = \tfrac12$. We then estimate
    \begin{align*}
      |\frakB_1[\Pi_h \fraka - \fraka](v,\hat e)| &\leq \| \Pi_h \fraka - \fraka \|_{L^r(\Omega)} \| \GRAD v \|_{\bL^{2+\epsilon}(\Omega)} \| \GRAD \hat e \|_{\bL^2(\Omega)} \\
        &\leq C \| \Pi_h \fraka - \fraka \|_{L^r(\Omega)}^2 \| \GRAD v \|_{\bL^{2+\epsilon}(\Omega)}^2 + \eta \| \GRAD \hat e \|_{\bL^2(\Omega)}^2.
    \end{align*}
    We may then choose $\eta$ as small as needed to absorb the second term on the left hand side of our main estimate \eqref{Eq: mainPerturbed}. This fixes the constant $C$. We conclude by recalling \eqref{eq:CoeffLrConvConductivity}, which shows that the remaining term vanishes in the limit.

    \item \textbf{The case $s \in (0,1):$}     
    Let now $(\fraka,v) \in \calZ^{s}$ and, at first, attempt to proceed to estimate as before, \ie by using a higher integrability result. Thus, assuming, for the time being, that there is $\epsilon >0$ such that $v \in W^{s,2+\epsilon}(\Omega_R)$ we estimate
    \begin{multline*}
      \left|\frakB_s[\Pi_h \fraka - \fraka](v,\hat e) \right| \lesssim \\ \iint_{\calD_R} \left|\frac{\Pi_h \frakA(x,y) - \frakA(x,y)}{|x-y|^{\tfrac{\epsilon n}{2(2+\epsilon)}}}\right| \left| \frac{v(x)-v(y)}{|x-y|^{\tfrac{n}{2+\epsilon}+s}} \right| \left| \frac{\hat e(x)- \hat e(y)}{|x-y|^{n/2+s}} \right| \diff x \diff y,
    \end{multline*}
    where we used the obvious notation $\Pi_h \frakA(x,y) = \tfrac12 \left( \Pi_h \fraka(x) + \Pi_h \fraka(y) \right)$. If we were to apply H\"older's inequality then, the estimate reduces to proving the convergence to zero of
    \[
      \iint_{\calD_R} \frac{ |\Pi_h \frakA(x,y) - \frakA(x,y)|^{\tfrac{2(2+\epsilon)}\epsilon}}{|x-y|^n} \diff x \diff y.
    \]
    However, due to the singular factor $|x-y|^{-n}$ this expression may not even be bounded! Instead, we argue as follows. Let $\tilde{f}, \tilde{v}$ denote the extensions by zero of $f$ and $v$ respectively. We then see that, because of the choice of domain of integration and the exterior condition in \eqref{eq:BLFormsConductivity}, the function $\tilde v \in H^s(\R^n)$ satisfies
    \[
      \gamma_{s,n,2}\iint_{\R^{2n} } \frakA(x,y) \frac{\tilde{v}(x)-\tilde{v}(y)}{|x-y|^s}  \frac{\varphi(x)-\varphi(y)}{|x-y|^s} \frac{ \diff x \diff y }{|x-y|^n} = \int_{\R^n} \tilde f(x) \varphi(x) \diff x
    \]
    for all $\varphi \in C_0^\infty(\R^n)$. We may then invoke the higher intregrability \emph{and} differentiability result in \cite[Theorem 1.1]{MR3336922} to assert that there are $\epsilon,\delta>0$ such that $\tilde v \in W^{s+\delta,2+\epsilon}_{\mathrm{loc}}(\R^n)$, and as a consequence $v \in W^{s+\delta,2+\epsilon}(\Omega_R)$. We may then proceed as follows. Set $\varepsilon = \min\{ \epsilon,\delta\}$ and $\tfrac1r + \tfrac1{2+\varepsilon} = \tfrac12$, then
    \begin{multline*}
      \left|\frakB_s[\Pi_h \fraka - \fraka](v,\hat e) \right|  \lesssim \\ \iint_{\calD_R} \left|\frac{\Pi_h \frakA(x,y) - \frakA(x,y)}{|x-y|^{\tfrac{n}r -\varepsilon}}\right| \left| \frac{v(x)-v(y)}{|x-y|^{\tfrac{n}{2+\varepsilon}+(s+\varepsilon)}} \right| \left| \frac{\hat e(x)- \hat e(y)}{|x-y|^{n/2+s}} \right| \diff x \diff y \\
      \lesssim \left(
      \iint_{\calD_R} \frac{ |\Pi_h \frakA(x,y) - \frakA(x,y)|^r} {|x-y|^{n-\varepsilon r } } \diff x \diff y \right)^{1/r}
          | v |_{W^{s+\varepsilon,2+\varepsilon}(\Omega_R)} |\hat e |_{\widetilde{H}^s(\Omega)}.
    \end{multline*}
    Since $n-\varepsilon r < n$ we have $|x-y|^{-n+\varepsilon r} \in L^1(\calD_R)$. In addition, $\fraka, \Pi_h \fraka \in \calH$, so that
    \[
      \frac{ |\Pi_h \frakA(x,y) - \frakA(x,y)|^r} {|x-y|^{n-\varepsilon r } } \leq \frac{(2a_{\max})^r}{|x-y|^{n-\varepsilon r } } \in L^1(\calD_R).
    \]
    Finally, owing to \eqref{eq:CoeffLrConvConductivity}, we have that (up to sub-sequences) $\Pi_h \frakA \to \frakA$ \mae in $\calD_R$. We may then apply the dominated convergence theorem to assert that this term goes to zero.
  \end{itemize}

  We comment that the, rather useful, higher smoothness result of \cite[Theorem 1.1]{MR3336922} is a feature exclusive to nonlocal problems; this is explored in \cite{MR3552254,MR3907738,MR3896636,MR3339179}.

\end{itemize}

Having verified all the assumptions regarding the discretization, it is a matter of translating the results in Sections~\ref{subsub:hToZero} and \ref{subsub:GammaConvergenceDiscrAbs}.

\begin{theorem}[convergence as $h \downarrow0$]
Fix $s \in (0,1]$ and assume the family of pairs $\{(\ofraka_h ,\ou_h)\} \in \calH_h \times V_h$ solve the discrete version of \eqref{eq:ObjConductivity}--\eqref{eq:StateConductivity}. Then there is $(\ofraka,\ou) \in \calZ^s$ such that, as $h\downarrow0$, and up to sub-sequences, the following hold:
\begin{enumerate}
  \item $\ofraka_h \rightharpoonup^* \ofraka$ in $L^\infty(\Omega_R)$.
  \item The pair $(\ofraka,\ou)$ solves \eqref{eq:ObjConductivity}--\eqref{eq:StateConductivity}.
  \item $J(\ofraka_h,\ou_h) \to J(\ofraka,\ou)$.
  \item $\ou_h \to \ou$ in $\widetilde{H}^s(\Omega)$.
  \item $\ofraka_h \to \ofraka$ in $L^q(\Omega)$.
\end{enumerate}
\end{theorem}
\begin{proof}
  This is Theorem~\ref{them:LimHtoZeroAbs} adapted to our setting.
\end{proof}

\begin{theorem}[convergence of minimizers]
Fix $h \in (0,h_0)$. Given a family of discrete optimal design coefficients $\{\ofraka_{s,h}\}_{s \in (0,1)} \subset \calH_h$, there is $\ofraka_h \in \argmin_{\fraka_h \in \calH_h} r_h^1(\fraka_h)$ such that, up to sub-sequences, $\ofraka_{s,h} \rightharpoonup^* \ofraka_h$ as $s \uparrow 1$. Moreover,
\[
  \lim_{s\uparrow1} r_h^s (\ofraka_{s,h}) = r_h^1(\ofraka_h).
\]
Moreover, setting $\ou_{s,h}= T_h^s(\ofraka_{s,h})$ and $\ou_h = T_h^1(\ofraka_h)$, we also obtain that
\[
  \lim_{s\uparrow1} \| \ou_h - \ou_{s,h} \|_{\widetilde{H}^s(\Omega)} = 0.
\]
In addition $\ou_{s,h} \to \ou_h$ in $L^2(\Omega)$. Finally, $\ofraka_{s,h} \to \ofraka_h$ in $L^q(\Omega)$.
\end{theorem}
\begin{proof}
This is Corollary~\ref{col:LimDelta0h>0}. Notice that, since we are dealing with finite dimensional spaces, many of the statements here may be easier to obtain independently.
\end{proof}

\subsection{Asymptotic compatibility}
\label{sub:ACConductivity}

We are now finally ready to verify the conditions that guarantee asymptotic compatibility of our discretization.

\begin{itemize}
  \item \textbf{Asymptotic approximation property:} Since, for every $s \in (0,1]$, we have $X_{s,h} = V_h$ it suffices, again, to set as $v_k$ the Scott-Zhang or Chen-Nochetto interpolant of the given function $v \in H^1_0(\Omega)$.

  \item \textbf{Asymptotic compatibility of state equations:} Let $\fraka \in \calH$ and $u = T^1(\fraka) \in H^1_0(\Omega)$. Assume the sequence $\{(s_k,h_k)\}_{k=1}^\infty \subset (0,1) \times (0, h_0)$ is such that, as $k \uparrow \infty$, we have $(s_k,h_k) \to (1,0)$. Let now $k \in \N$ and, for simplicity, set $\fraka_k = \Pi_{h_k} \fraka$ and $u_k = T_{h_k}^{s_k}(\fraka_k) = \hat{\calG}_{s_k,h_k}[\fraka]u$. We must show that, as $k \uparrow \infty$,
  \[
    \| u-u_k \|_{L^2(\Omega)} \to 0.
  \]

  We will show this indirectly by observing the following. An a priori estimate on $u_k$ guarantees that
  \[
    \| u_k \|_{\widetilde{H}^{s_k}(\Omega)} \lesssim \| f \|_{L^2(\Omega)},
  \]
  with a constant independent of $k$. The compactness result of \cite[Theorem 4]{MR3586796} then guarantees the existence of $w \in H^1_0(\Omega)$ such that $u_k \to w$ in $L^2(\Omega)$. We will then show $u=w$, by showing that $w$ minimizes the energy $E_1[\fraka]$ on $H^1_0(\Omega)$. If this is the case, we must necessarily have $u=w$, and the desired asymptotic compatibility will follow.

  For $k \in \N$ we define the Borel measures $\nu_k$ and $\nu$ as
  \[
    \nu_k(S) = \gamma_{s_k,n,2} \int_S \int_{\Omega_R} \frac{|u_k(x) -u_k(y)|^2}{|x-y|^{2s_k}} \frac{\diff y \diff x}{|x-y|^n},
    \qquad
    \nu(S) = \int_S |\GRAD w(x)|^2 \diff x,
  \]
  for any measurable $S \subset \Omega$. Owing to \cite{MR2041005}, see also \cite{MR4201703}, we have
  \[
    \nu(S) \leq \liminf_{k \uparrow \infty} \nu_k(S).
  \]
  Recall now that $\fraka_k \to \fraka$ in any $L^q(\Omega)$ for $q\in (1, \infty)$ so that, by passing to a sub-sequence, we have that $\fraka_k \to \fraka$ almost everywhere. The generalized Fatou lemma presented in \cite[Lemma 2.2]{MR705462} and \cite[Theorem 2.2.(i)]{MR1768500} then implies
  \[
    \int_{\Omega_R} \fraka(x) \diff \nu(x) \leq \liminf_{k \uparrow \infty} \int_{\Omega_R} \fraka_k(x) \diff \nu_k(x).
  \]
  Observe now that
  \[
    \int_{\Omega_R} \fraka(x) \diff \nu(x) = \int_\Omega \fraka(x) |\GRAD w(x)|^2 \diff x,
  \]
  and
  \[
    \int_{\Omega_R} \fraka_k(x) \diff \nu_k(x) = \int_{\Omega_R} \fraka_k(x) \int_{\Omega_R} \frac{|u_k(x) -u_k(y)|^2}{|x-y|^{2s_k}} \frac{\diff y \diff x}{|x-y|^n} = \frakB_{s_k}[\fraka_k](u_k,u_k),
  \]
  where the last equality is obtained by symmetrization. In conclusion,
  \begin{equation}
  \label{eq:Ponce}
    \int_\Omega \fraka(x) |\GRAD w(x)|^2 \diff x \leq \liminf_{k \uparrow \infty} \frakB_{s_k}[\fraka_k](u_k,u_k),
  \end{equation}
  and, since $u_k \to w$ in $L^2(\Omega)$,
  \[
    E_0[\fraka](w) \leq \liminf_{k \uparrow \infty} E_{s_k}[\fraka_k](u_k) \leq \liminf_{k \uparrow \infty} E_{s_k}[\fraka_k](v_k),
  \]
  where $v_k \in V_{h_k}$ is arbitrary. Let now $v \in H^1_0(\Omega)$ be arbitrary and choose $v_k$ so that $v_k \to v$ in $H^1_0(\Omega)$. This can be achieved, as it was already mentioned, by using the Scott-Zhang or the Chen-Nochetto interpolant. Moreover, it implies that
  \[
    \liminf_{k \uparrow \infty} E_{s_k}[\fraka_k](v_k) = \frac12 \liminf_{k \uparrow \infty}\frakB_{s_k}[\fraka_k](v_k,v_k) - \int_\Omega f(x) v(x) \diff x.
  \]

  The next step is to observe that the quantity
  \[
    \interleave \varphi \interleave_k = \frakB_{s_k}[\fraka_k](\varphi,\varphi)^{1/2}
  \]
  defines a norm on $\widetilde{H}^{s_k}(\Omega)$. Thus,
  \[
    \interleave v_k\interleave_k \leq \interleave v \interleave_k + \interleave v - v_k\interleave_k.
  \]
  Now
  \[
    \interleave v \interleave_k^2 =  \int_{\Omega_R} \fraka_k(x)\left[  \gamma_{s_k,n,2}\int_{\Omega} \frac{ |v(x) - v(y)|^2 }{|x-y|^{n+2s_k} } \diff y \right]\diff x.
  \]
  Owing to \cite[Corollary 1]{MR3586796}
  \[
    \gamma_{s_k,n,2}\int_{\Omega} \frac{ |v(x) - v(y)|^2 }{|x-y|^{n+2s_k} } \diff y \to |\GRAD v(x)|^2
  \]
  in $L^1(\Omega)$, which combined with $\fraka_k \rightharpoonup^* \fraka$ in $L^\infty(\Omega)$ gives
  \[
    \lim_{k \uparrow \infty} \interleave v \interleave_k^2 = \int_\Omega \fraka(x) |\GRAD w(x)|^2 \diff x.
  \]
  Finally,
  \begin{align*}
    \interleave v - v_k \interleave_k^2 &= \gamma_{s_k,n,2}\iint_{\calD_R} \frakA_k(x,y) \frac{|(v-v_k)(x) - (v-v_k)(y)|^2}{|x-y|^2} \frac{ \diff x \diff y}{|x-y|^n} \\
    &\leq a_{\max} \gamma_{s_k,n,2} \iint_{\calD_R}\frac{|(v-v_k)(x) - (v-v_k)(y)|^2}{|x-y|^2} \frac{ \diff x \diff y}{|x-y|^n} \\
    &= a_{\max} |v-v_k|_{\widetilde{H}^{s_k}(\Omega)}^2 \leq a_{\max} C \| \GRAD (v-v_k) \|_{\bL^2(\Omega)}^2,
  \end{align*}
  where owing to \cite[Remark 5]{MR3586796} the constant $C$ is independent of $k$. In conclusion $\interleave v - v_k \interleave_k \to 0$ as $k \uparrow \infty$.

  In summary,
  \[
    E_0[\fraka](w) \leq \frac12 \liminf_{k \uparrow \infty}\frakB_{s_k}[\fraka_k](v_k,v_k) - \int_\Omega f(x) v(x) \diff x \leq E_0[\fraka](v),
  \]
  as desired.

\end{itemize}

We have now verified all the requisite assumptions for asymptotic compatibility.

\begin{theorem}[asymptotic compatibility]
\label{thm:ACConductivity}
Let $\{(\ofraka_{s,h},\ou_{s,h})\}_{s \in (0,1), h \in (0,h_0)}$ be a family of solutions to problem \eqref{eq:ObjConductivity}--\eqref{eq:StateConductivity}. This family is asymptotically compatible in the sense of Definition~\ref{def:ACAbs}. Moreover, we additionally have,
\[
  \lim_{k \uparrow \infty} J(\ofraka_{s_k,h_k},\ou_{s_k,h_k}) = J(\ofraka,\ou),
\]
and
\[
  \lim_{k\uparrow \infty} \| \ou - \ou_{s_k,h_k} \|_{\widetilde{H}^{s_k}(\Omega)} = 0.
\]
\end{theorem}
\begin{proof}
  This is Theorem~\ref{thm:ACAbs} adapted to our setting.
\end{proof}

\section{Optimal design in peridynamics}
\label{sec:Peridynamics}

As a second application of the general framework developed in Section~\ref{sec:AbstractSetting} we study an optimal design problem in peridynamics. Many of the ideas and developments follow those of the conductivity problem of Section~\ref{sec:Conductivity}. Thus, we shall be very concise.

\subsection{Setup}
\label{sub:SetupPeridynamics}

We now describe our problem. As before, $s \in (0,1]$ and $\delta = 1-s$.

\begin{itemize}
  \item We keep $0<a_{\min} \leq a_{\max}$, and
  \[
    \calH = \left\{ a \in L^\infty(\Omega_R) \ \middle| \ a_{\min} \leq a(x) \leq a_{\max} \mae \Omega \right\}.
  \]
  The metric $d$, as before, is the metrization of weak-* convergence over bounded subsets of $L^\infty(\Omega_R)$.

  \item In this case $X_1 = \bH^1_0(\Omega)$, $X_s= \widetilde{\bH}^{s}(\Omega)$ for $s \in (0,1)$, and $X_0 = \bL^2(\Omega)$. We keep the normalization constants so that this scale satisfies a uniform embedding property.

  \item The bilinear forms that define the state equations are:
  \begin{equation}
  \label{eq:BLFormsPeridyn}
    \begin{aligned}
      \frakB_{s}[\fraka](\bv,\bw) &= \gamma_{s,n,2}\iint_{\calD_R} \frakA(x,y) \frac{ \Dxy\bv(x,y) }{|x-y|^s} \frac{\Dxy \bw(x,y) }{|x-y|^s} \frac{\diff x \diff y}{|x-y|^n}, \\
      \frakB_1[\fraka](\bv,\bw) &= \frac1{d+2}\int_\Omega \fraka(x) \left( 2\beps(\bw)(x) \colon \beps(\bv)(x) + \DIV \bw(x) \DIV \bv(x) \right)\diff x, \\
    \end{aligned}
  \end{equation}
  where $\beps$ denotes the symmetric gradient, $\colon$ is the Frobenius inner product, and
  \[
    \Dxy \bw(x,y) = \left( \bw(x) - \bw(y) \right) \cdot \frac{x-y}{|x-y|}
  \]
  is the projected difference. As before, if $\fraka \in \calH$ and $x,y \in \Omega_R$, we set $\frakA(x,y) = \tfrac12 \left( \fraka(x) + \fraka(y) \right)$.
   
  \item As before, for  $q \in (1,\infty)$ and strictly positive $\Lambda \in L^\infty(\Omega)$ we define
  \[
    \phi(\fraka) = \int_\Omega \Lambda(x) |\fraka(x)|^q \diff x.
  \]
  We recall that $\phi$ is lower semicontinuous with respect to the metric $d$.
\end{itemize}

The optimal design problem in peridynamics is then stated as follows: Let $s \in (0,1]$, $q \in (1,\infty)$, and $\bef \in \bL^2(\Omega)$. We seek to find a pair $(\fraka,\bu) \in \calH \times \widetilde{\bH}^s(\Omega)$ that minimizes
\begin{equation}
\label{eq:ObjPeridyn}
  J(\fraka,u) = \int_\Omega \bef(x) \cdot \bu(x) \diff x + \int_\Omega \Lambda(x) |\fraka(x) |^q \diff x,
\end{equation}
subject to
\begin{equation}
\label{eq:StatePeridyn}
  \frakB_{s}[\fraka](\bu,\bv) = \int_\Omega \bef(x) \cdot \bv(x) \diff x, \qquad \forall \bv \in \widetilde{\bH}^s(\Omega).
\end{equation}
Here, for uniformity, we have set $\widetilde{\bH}^1(\Omega) = \bH^1_0(\Omega)$.

\subsection{Verification of structural assumptions and results}
\label{sub:AnalysisPeridynamics}

We now, once again, verify the assumptions of our abstract framework.

\begin{itemize}
  \item \textbf{Uniform embedding:} This is as before.

  \item \textbf{Asymptotic compactness:} This is, again, as in the conductivity case.

  \item \textbf{Strong continuity:} This can be obtained \emph{mutatis mutandis} the arguments for the conductivity case.

  \item \textbf{Uniform boundedness and coercivity:} The argument for uniform boundedness repeats the computations of the conductivity case. However, the argument for uniform coercivity needs to be supplemented with a \emph{uniform} (in $s$) fractional Korn's inequality that bounds $|\bv|^{2}_{ \widetilde{\bH}^s(\Omega)}$ by $ \frakB_{s}[1](\bv,\bv) $. The issue here is that the bilinear form depends only on $\bv$ via the projected difference. The fractional  Korn's inequality  proved in \cite{Mengesha2019, ScottMengesha2019b} does not make clear the dependence of constants in $s$. Let us then clarify the dependence of the constant in \cite[Theorem 4.2 ]{Mengesha2019} in $s$. To do so, we denote by $\widehat\bv$ the Fourier transform of $\bv$. Following \cite{Mengesha2019}, we have 
  \begin{equation*}
  \begin{aligned}
    \frakB_{s}[1](\bv,\bv)  &= \gamma_{n, s, 2} \iint_{\mathbb{R}^{n}}  \frac{ |\Dxy\bv(x,y)|^{2} }{|x-y|^{n+2s}} \diff x \diff y \\
    & =\gamma_{n, s, 2}   \int_{ \mathbb{R}^{n} }  \frac{1}{ |h|^{n+2s}}  \int_{\mathbb{R}^{n}} \left| (\bv(x + h) - \bv(x) ) \cdot \frac{h}{|h|}  \right|^{2} \diff x \diff h\\
    & = 2\gamma_{n, s, 2}   \int_{ \mathbb{R}^{n} } \frac{1}{ |h|^{n+2s}} \int_{\mathbb{R}^{n}} \left |\widehat{\bv}(\xi)\cdot  \frac{h}{|h|}\right|^2 (1-\cos (2\pi \xi\cdot h)) \diff \xi \diff h\\
    &\geq \gamma_{n, s, 2}  c(n) \Gamma(1-s) \int_{ \mathbb{R}^{n}} |\xi|^{2s}|\widehat{\bv}(\xi)|^{2} \diff \xi \\
    &=\gamma_{n, s, 2} c(n) \int_{ \mathbb{R}^{n} }  \frac{1}{ |h|^{n+2s}}  \int_{\mathbb{R}^{n}} \left| \bv(x + h) - \bv(x) \right|^{2} \diff x \diff h \\
    &=  c(n) |\bv|^{2}_{{\bH}^{s}(\mathbb{R}^{d})} = c(n)  | \bv|^{2}_{ \widetilde{\bH}^s(\Omega)},
  \end{aligned}
  \end{equation*}
  where by $c(n)$ we indicated a constant that depends on the spatial dimension $n$ only.

  \item \textbf{Parametric continuity:} Lengthy, but trivial, computations that mimic the conductivity case imply this result. For this, the explicit values of the normalization constants are important; see \cite[Appendix A]{MR3424902}.

  \item \textbf{Lower semicontinuity:} See \cite[Theorem 8]{MR4629861}.
\end{itemize}

Since these assumptions are verified we immediately obtain a series of results regarding our problem.

\begin{lemma}[uniform well posedness]
For every $s \in (0,1]$ and each coefficient $\fraka \in \calH$ there is a unique $\bu_s \in \widetilde{\bH}^s(\Omega)$ that solves \eqref{eq:StatePeridyn}. In addition, this solution satisfies
\[
  |\bu_s|_{\widetilde{\bH}^s(\Omega)} \lesssim \| \bef \|_{\bL^2(\Omega)},
\]
where the implicit constant is independent of $s$.
\end{lemma}
\begin{proof}
This is Lemma~\ref{lem:StatesAreOKAbs} in the present setting.
\end{proof}

As before, the previous result allows us to define the design-to-state mapping $\bT^s : \calH \to \widetilde{\bH}^s(\Omega)$ and the reduced costs $\{r^s\}_{s \in (0,1]}$.

\begin{theorem}[existence]
For every $s\in (0,1]$, problem \eqref{eq:ObjPeridyn}--\eqref{eq:StatePeridyn} has a solution.
\end{theorem}
\begin{proof}
This specializes Theorem~\ref{thm:ExistenceAbs}.
\end{proof}

We now show the variational convergence of these problems as $s \uparrow 1$.

\begin{theorem}[convergence of minimizers]
Let $\{\ofraka_s\}_{s \in (0,1)}$ be a family of minimizers of the reduced costs $\{r^s\}_{s \in (0,1)}$. In other words, the family of pairs $\{(\ofraka_s, \obu_s = \bT^s(\ofraka_s))\}_{s \in (0,1)}$ solve problem \eqref{eq:ObjPeridyn}--\eqref{eq:StatePeridyn}. Then, as $s \uparrow 1$,
\begin{itemize}
  \item The family $\{r^s\}_{s \in (0,1)}$ $\Gamma$-converges to $r^1$ in the metric induced by weak-* convergence in $L^\infty(\Omega)$ over bounded sets. Moreover, this family is equicoercive.

  \item There is $\ofraka_1 \in \calH$ such that, up to subsequences, $\ofraka_s \rightharpoonup^* \ofraka_1$ in $L^\infty(\Omega)$. Moreover,
  \[
    \lim_{s \uparrow 1} r^s(\ofraka_s) = r^1(\ofraka_1), \qquad \lim_{s\uparrow1} \| \obu_s - \obu_1 \|_{\widetilde{\bH}^s(\Omega)} = 0,
  \]
 and we have $\obu_s \to \obu_1$ in $\bL^2(\Omega)$. Finally, $\ofraka_s \to \ofraka_1$ in $L^q(\Omega)$ for every $q<\infty$, and $(\ofraka_1,\obu_1)$ is a solution to \eqref{eq:ObjPeridyn}--\eqref{eq:StatePeridyn} for $s = 1$.
\end{itemize}
\end{theorem}
\begin{proof}
This result combines Theorem~\ref{thm:GconvAsDeltaTo0} and Corollary~\ref{col:LimDelta0}.
\end{proof}

\subsection{Discretization and results}

As in the conductivity case, we set $X_{s,h} = \bV_h = V_h^n$ and $\calH_h = \calL_0^{-1} (\Triang_h) \cap \calH$. We now verify the assumptions.

\begin{itemize}
  \item \textbf{Discrete embedding and approximation property:} Same as the scalar case. Indeed, it suffices to use $n$-copies of a suitable interpolant.

  \item \textbf{Density:} Requires no changes, and the same observations as in the conductivity case apply.

  \item \textbf{Perturbed Galerkin projection:} One can follow, without requiring any change, the conductivity case up to the point where higher integrability and differentiability results are required. Once we reach that point, we again argue the cases $s=1$ and $s \in (0,1)$ separately.
  \begin{itemize}
    \item \textbf{The case $s=1$:} We need a Meyers-type higher integrability result for the system of elasticity. Results of this kind are plentiful in the literature. For instance, \cite{MR1607245} proves this for a nonlinear version of the Stokes problem; for stronlgy elliptic systems satisfying the Legendre-Hadamard condition, like linear elasticity, this is shown in \cite[Theorem 2]{MR417568}. For completeness and clarity, we provide a direct proof in Appendix~\ref{sec:Appendix}.

    \item \textbf{The case $s \in (0,1)$:} In this case one needs a higher integrability and differentiability results for the solution of the system of peridynamics. This is provided, for instance, in \cite[Remark 3.5]{MR3896636}.
  \end{itemize}
\end{itemize}

With all these assumptions being verified, all the pertinent results about discretization expediently follow.

\begin{theorem}[convergence as $h \downarrow0$]
Fix $s \in (0,1]$ and assume the family of pairs $\{(\ofraka_h ,\obu_h)\} \in \calH_h \times \bV_h$ solve the discrete version of \eqref{eq:ObjPeridyn}--\eqref{eq:StatePeridyn}. Then there is $(\ofraka,\obu) \in \calZ^s$ such that, as $h\downarrow0$, and up to sub-sequences, the following hold:
\begin{enumerate}
  \item $\ofraka_h \rightharpoonup^* \ofraka$ in $L^\infty(\Omega_R)$.
  \item The pair $(\ofraka,\obu)$ solves \eqref{eq:ObjPeridyn}--\eqref{eq:StatePeridyn}.
  \item $J(\ofraka_h,\obu_h) \to J(\ofraka,\obu)$.
  \item $\obu_h \to \obu$ in $\widetilde{\bH}^s(\Omega)$.
  \item $\ofraka_h \to \ofraka$ in $L^q(\Omega)$ for any $q \in [1, \infty)$.
\end{enumerate}
\end{theorem}
\begin{proof}
  This is Theorem~\ref{them:LimHtoZeroAbs}.
\end{proof}

\begin{theorem}[convergence of minimizers]
Fix $h \in (0,h_0)$. Given a family of discrete optimal design coefficients $\{\ofraka_{s,h}\}_{s \in (0,1)} \subset \calH_h$, there is $\ofraka_h \in \argmin_{\fraka_h \in \calH_h} r_h^1(\fraka_h)$ such that,  as $s \uparrow 1$ and up to sub-sequences, $\ofraka_{s,h} \rightharpoonup^* \ofraka_h$ in $L^\infty(\Omega)$. Moreover,
\[
  \lim_{s\uparrow1} r_h^s (\ofraka_{s,h}) = r_h^1(\ofraka_h),
\]
and setting $\obu_{s,h}= \bT_h^s(\ofraka_{s,h})$ and $\obu_h = \bT_h^1(\ofraka_h)$, we also obtain that
\[
  \lim_{s\uparrow1} \| \obu_h - \obu_{s,h} \|_{\widetilde{\bH}^s(\Omega)} = 0.
\]
In addition $\obu_{s,h} \to \obu_h$ in $\bL^2(\Omega)$. Finally, $\ofraka_{s,h} \to \ofraka_h$ in $L^q(\Omega)$.
\end{theorem}
\begin{proof}
This is Corollary~\ref{col:LimDelta0h>0}.
\end{proof}

\subsection{Asymptotic compatibility}

Finally, we are ready to verify the assumptions that lead to asymptotic compatibility.

\begin{itemize}
  \item \textbf{Asymptotic approximation property:} This is proved, again, by using a suitable vector-valued interpolant.

  \item \textbf{Asymptotic compatibility of state equations:} One can follow the proof of the conductivity case with minor changes. Namely, in this case, we must define the Borel measures
  \begin{align*}
    \nu_k(S) &= \gamma_{s_k,n,2}\int_S \int_{\Omega_R} \frac{|\Dxy \bu_k(x,y)|^2}{|x-y|^{2s_k}} \frac{ \diff x \diff y}{|x-y|^n}, \\
    \nu(S) &= \frac1{n+2} \int_S \left( 2 |\beps(\bw)(x)|^2 + |\DIV \bw(x)|^2 \right) \diff x.
  \end{align*}
  One can then invoke \cite[Corollary 3.7]{MR3424902} to assert that
  \[
    \nu(S) \leq \lim_{k \uparrow \infty} \nu_k(S).
  \]
  The rest of the argument follows with no significant modifications.
\end{itemize}

In summary, we can prove asymptotic compatibility for our problem related to peridynamics.

\begin{theorem}[asymptotic compatibility]
Assume that each member of the family
\[
  \{(\ofraka_{s,h},\obu_{s,h})\}_{s \in (0,1), h \in (0,h_0)}
\]
is a solution to problem \eqref{eq:ObjPeridyn}--\eqref{eq:StatePeridyn} for the corresponding parameter values. This family is asymptotically compatible in the sense of Definition~\ref{def:ACAbs}. Moreover, we additionally have
\[
  \lim_{k \uparrow \infty} J(\ofraka_{s_k,h_k},\obu_{s_k,h_k}) = J(\ofraka,\obu),
\]
and
\[
  \lim_{k\uparrow \infty} \| \obu - \obu_{s_k,h_k} \|_{\widetilde{\bH}^{s_k}(\Omega)} = 0.
\]
\end{theorem}
\begin{proof}
  This is Theorem~\ref{thm:ACAbs}.
\end{proof}

\section{Numerical illustrations}
\label{sec:Numerics}

To illustrate and complement our theoretical developments here we present some numerical simulations of the optimal design problems of Section~\ref{sec:Conductivity}. The implementation was carried out using the PyNucleus library \cite{PyNucleus}. For definiteness we set $\Lambda \equiv \tfrac12$ and $q = 2$. Thus, the objective functional reads
\begin{equation}
\label{eq:ObjectiveNumExp}
  J(\fraka,u) = \int_\Omega f(x) u(x) \diff x + \frac12 \| \fraka \|_{L^2(\Omega)}^2.
\end{equation}
The discrete reduced cost is then defined as
\begin{equation}
\label{eq:ReducedCostNumExp}
  r_h^s(\fraka) = \int_\Omega f(x) T_h^s[\fraka](x) \diff x + \frac12 \| \fraka \|_{L^2(\Omega)}^2,
\end{equation}
where $T_h^s$ denotes the discrete design to state mapping.

\subsection{Projected gradient descent}
To find stationary points of \eqref{eq:ReducedCostNumExp} we proceed via a projected descent algorithm. Starting from $\fraka_0 = \tfrac12( a_{\min} + a_{\max} )$ we iterate via
\begin{equation}
\label{eq:DescentPrev}
  \frac{ \widetilde{\fraka}_{k+1} - \fraka_k } \tau =- \Dxy r_h^s[\fraka_k], \qquad \fraka_{k+1} = \underset{\calH_h}{ \Pr} \ \widetilde{\fraka}_{k+1}.
\end{equation}
Here $\tau$ is a user-defined parameter (the step size), $\Pr_{\calH_h}$ is the $L^2$-projection onto the set $\calH_h$, and $\Dxy r_h^s$ denotes the Gateaux derivative of the reduced cost, which is obtained in the following result.

\begin{lemma}[$\Dxy r_h^s$]
\label{lem:DerivR}
For any $\fraka \in \mathring\calH$ the Gateaux derivative of $r_h^s$, defined in \eqref{eq:ReducedCostNumExp}, in the direction $\frakb \in L^\infty(\Omega)$ is given by
\[
  \Dxy r_h^s[\fraka](\frakb) = \int_\Omega f(x) \Dxy T_h^s[\fraka](\frakb) \diff x + \int_\Omega \fraka(x) \frakb(x) \diff x,
\]
where $\Dxy T_h^s[\fraka](\frakb)$ denotes the Gateaux derivative of the design to state mapping. The function $z_h^s = \Dxy T_h^s[\fraka](\frakb) \in V_h$ is the solution to
\[
  \frakB_s[\fraka](z_h^s,v_h) = -\frakB_s[\frakb]( T_h^s[\fraka],v_h) , \qquad \forall v_h \in V_h
\]
\end{lemma}
\begin{proof}
Clearly, for $\epsilon>0$,
\[
  \frac{r_h^s[\fraka+ \epsilon \frakb] - r_h^s[\fraka]}\epsilon  = \int_\Omega f(x) \frac{T_h^s[\fraka + \epsilon \frakb] - T_h^s[\fraka]}\epsilon \diff x + \int_\Omega \fraka(x) \frakb(x) \diff x + \frac\epsilon2 \| \frakb\|_{L^2(\Omega)}^2.
\]
Sending $\epsilon \downarrow 0$ yields the desired expression for $\Dxy r_h^s$, provided $T_h^s$ is Gateaux differentiable.

To compute $\Dxy T_h^s$ we first show its continuity. To see this we let $u_h = T_h^s[\fraka] \in V_h$ and observe that, since $\fraka \in \mathring{\calH}_h$, for $\epsilon$ sufficiently small the function $u_\epsilon = T_h^s[\fraka + \epsilon \frakb] \in V_h$ is well-defined. Next we compute
\begin{align*}
  a_{\min} | u_h - u_\epsilon |_{\widetilde{H}^s(\Omega)}^2 &\leq \frakB_s[\fraka](u_h - u_\epsilon, u_h - u_\epsilon) \\
  &= \int_\Omega f(x) (u_h - u_\epsilon)(x) \diff x - \frakB_s[\fraka]( u_\epsilon, u_h - u_\epsilon) \\
  &= \frakB_s[\fraka + \epsilon \frakb](u_\epsilon, u_h - u_\epsilon)- \frakB_s[\fraka]( u_\epsilon, u_h - u_\epsilon) \\
  &= \epsilon \frakB_s[\frakb](u_\epsilon, u_h - u_\epsilon),
\end{align*}
where we used the linearity of the mapping $\fraka \mapsto \frakB_s[\fraka]$. Since the family $\{ u_\epsilon\}_{\epsilon>0}$ can be bounded independently of $\epsilon$, we have the continuity of $T_h^s$.

Let $\epsilon$ be sufficiently small and denote $z_\epsilon = \tfrac1\epsilon ( T_h^s[\fraka+ \epsilon \frakb] - T_h^s[\fraka] )$. The previous computations also show that
\[
  \frakB_s[\fraka](z_\epsilon, v_h) = - \frakB_s[\frakb](u_\epsilon,v_h), \qquad \forall v_h \in V_h.
\]
Passing to the limit $\epsilon \downarrow 0$, we see that $z_\epsilon \to z_h^s = \Dxy T_h^s[\fraka](\frakb)$.
\end{proof}

\subsection{Implementation details}
With the identifications of Lemma~\ref{lem:DerivR} the projected gradient descent method \eqref{eq:DescentPrev} may be rewritten, for $\frakb \in \calL_0^{0}(\Triang_h)$, as
\begin{equation}
\label{eq:DescentMid}
  \int_\Omega \frac{ \widetilde{\fraka}_{k+1}(x) - \fraka_k(x)}\tau \frakb(x) \diff x =  - \int_\Omega f(x) z_k(x) \diff x - \int_\Omega \fraka_k(x) \frakb(x) \diff x
\end{equation}
with $z_k = \Dxy T_h^s[\fraka_k]$, \ie
\[
  \frakB_s[\fraka_k](z_k,v_h) = - \frakB_s[\frakb](T_h^s[\fraka_k],v_h), \qquad \forall v_h \in V_h.
\]
Set now $v_h = T_h^s[\fraka_k]$ and use the symmetry of the bilinear form to infer
\[
  \int_\Omega f(x) z_k(x) \diff x = \frakB_s[\fraka_k](z_k, T_h^s[\fraka_k]) = - \frakB_s[\frakb](T_h^s[\fraka_k], T_h^s[\fraka_k]).
\]
In summary, the descent algorithm may be rewritten as follows. Let $\fraka_0 = \tfrac12(a_{\min} + a_{\max} )$. For $k \geq 0$:
\begin{enumerate}
  \item Compute $u_k = T_h^s[\fraka_k]$, \ie
  \[
    \frakB_s[\fraka_k](u_k, v_h) = \int_\Omega f(x) v_h(x) \diff x, \qquad \forall v_h \in V_h.
  \]

  \item Let $\widetilde{\fraka}_{k+1} \in \calL_0^0(\Triang_h)$ solve, for every $\frakb \in \calL_0^0(\Triang_h)$,
  \[
    \int_\Omega \widetilde{\fraka}_{k+1}(x) \frakb(x) \diff x = \tau \frakB_s[\frakb](u_k,u_k) +(1-\tau) \int_\Omega \fraka_k(x) \frakb(x) \diff x.
  \]

  \item The next iterate $\fraka_{k+1}$ is:
  \[
    \fraka_{k+1} = \Pr_{\calH_h} \widetilde{\fraka}_{k+1}.
  \]
\end{enumerate}
Several comments about this algorithm are in order.
\begin{itemize}
  \item At each step, to compute $u_k$ the assembly and inversion of a stiffness matrix is needed.

  \item Owing to the fact that $\calH_h$ consists of piecewise constants, one only needs to test the equation for $\widetilde{\fraka}_{k+1}$ against the characteristic of each element, \ie we set $\frakb = \chi_T$ with $T \in \Triang_h$ and obtain
  \[
    \widetilde{\fraka}_{k+1|T} = \frac\tau{|T|} \frakB_s[\chi_T](u_k,u_k) + (1-\tau) \fraka_{k|T}.
  \]
  The matrices $\{ \frakB_s[\chi_T](\cdot,\cdot)\}_{T \in \Triang_h}$ can then be assembled, once and in parallel, at the beginning of the iterative procedure and then utilized at every iteration.

  \item This preprocessing step is extremely memory-demanding, and it is the most computationally intensive part of our method. In fact, to obtain the results we present below, we needed to fully employ our computational server: Intel Xeon Gold 6246R CPU running at 3.40 GHz with a total of 640 GB of RAM and 28 GB of swap space. This machine was running Fedora release 36 configured with Linux kernel 6.2.13. A multiprocessing pool of size 20, based on Python version 3.10.11, was used to assemble these matrices in parallel. The assembly and solving of linear systems, as mentioned above, was carried out using PyNucleus version 1.0. For reference, a small nonlocal problem with $\dim V_h = 961$ requires $19933.23 s \approx 5.5h$ to finalize.

  \item Since $\calH_h$ consists of piecewise constants, the projection can be easily defined element-wise via
  \[
    \Pr_{\calH_h} \frakb_{|T} = \max\left\{ \min\{ a_{\max}, \frakb_{|T} \}, a_{\min} \right\}, \qquad \forall T \in \Triang_h.
  \]
  
  \item The code written is publicly available in a GitHub repository, located at \cite{SiktarCode}.
\end{itemize}

\subsection{Results}
We let $n=2$, $\Omega = B_1$, $a_{\min} = 0.1$, $a_{\max} = 2.0$. The value of $R$ is specified in each experiment. The stepsize is set to $\tau = \tfrac14$, and the projected descent algorithm is run for the specified number of iterations.

\begin{figure}
  \begin{center}
    \includegraphics[scale = 0.51]{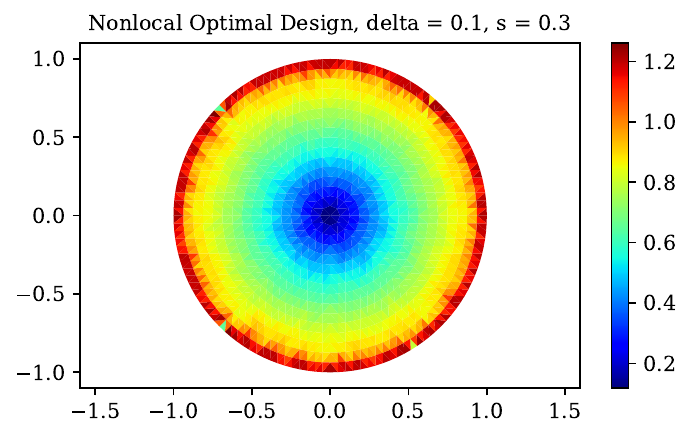}
    \includegraphics[scale = 0.5]{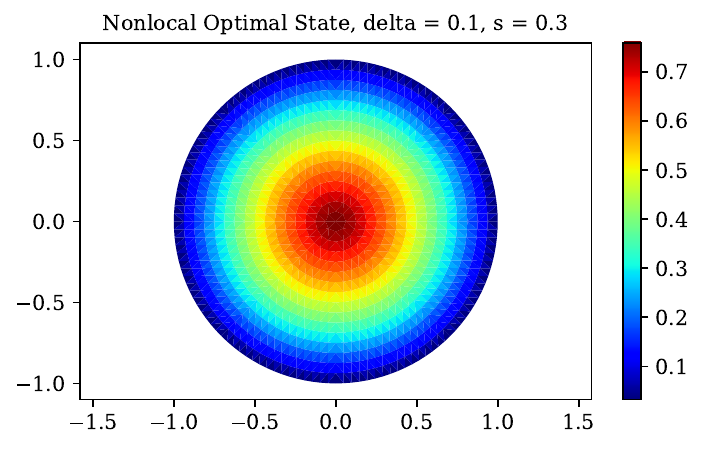}

    \includegraphics[scale = 0.51]{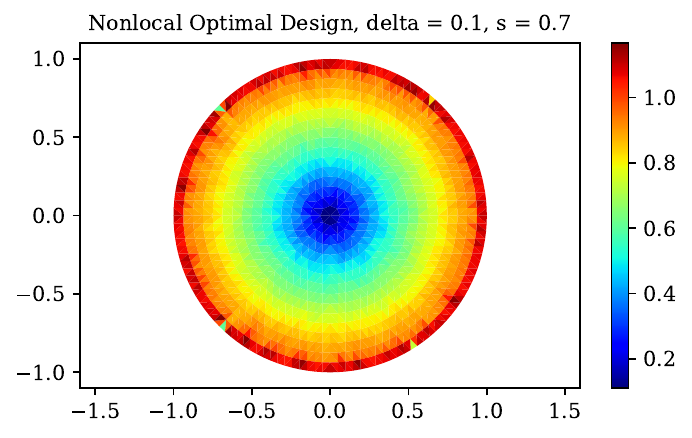}
    \includegraphics[scale = 0.5]{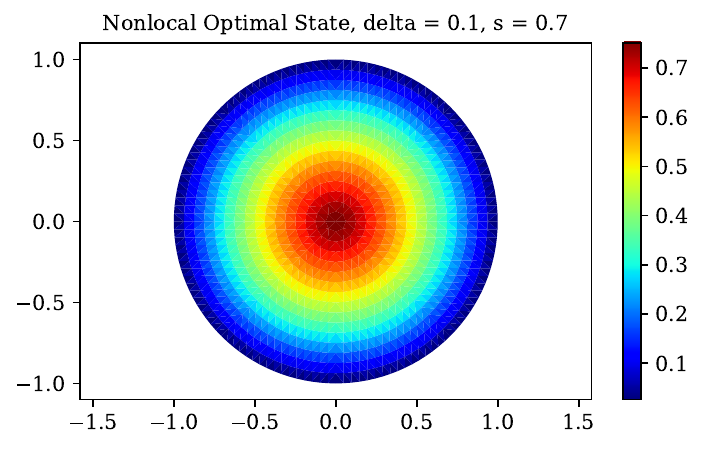}

    \includegraphics[scale = 0.51]{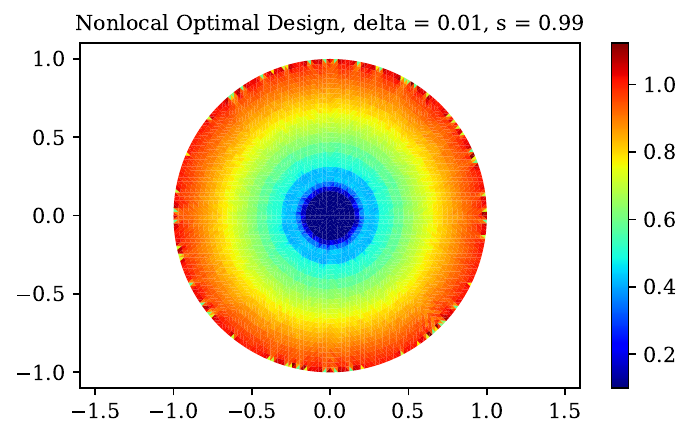}
    \includegraphics[scale = 0.5]{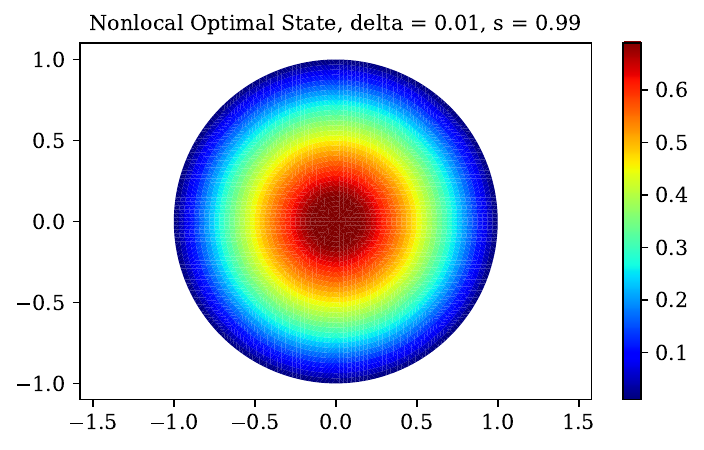}

    \includegraphics[scale = 0.51]{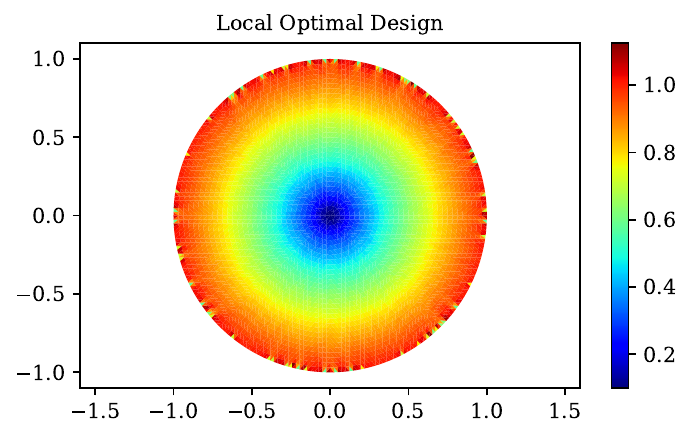}
    \includegraphics[scale = 0.5]{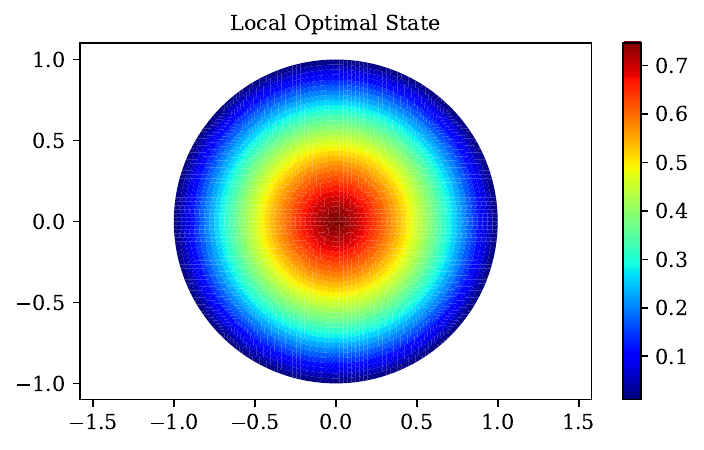}
  \end{center}
  \caption{Optimal design and state for $f \equiv 1$ and different values of $s$ and $R$. The value of $s$ is given in each plot, where $R$ is denoted as \textsf{\textup{delta}}. Notice that the last row corresponds to the local problem.}
  \label{fig:ConstRHSNL}
\end{figure}

\begin{table}
  \begin{center}
    \begin{tabular}{|| c | c | c | c | c | c | c | c||}
      \hline
      DOFs & Its & $s$ & $R$ & $\| \ou_{s,h} \|_{L^2(\Omega)}$ & $[\ou_{s,h} ]_{\widetilde{H}^s(\Omega)}$ & $\| \ofraka_{s,h} \|_{L^2(\Omega)}$ & $r_h^{s}(\ofraka_{s, h})$ \\
      \hline\hline
      $961$   & $50$  & $0.3$   & $0.1$   & $0.669808$  & $2.227475$ & $1.421823$ & $2.027113$ \\
      $961$   & $50$  & $0.7 $  & $0.1$   & $0.653769$  & $2.202400$ & $1.396489$ & $1.956274$\\
      $3969$  & $150$ & $0.99$  & $0.01$  & $0.635624$  & $2.168850$ & $1.370081$ & $1.883284$ \\
      $3969$  & $20$  & $1$     & $0$     & $0.636809$  & $2.175644$ & $1.370092$  & $1.882695$ \\
     \hline
    \end{tabular}
  \end{center}
\caption{Experimental data for optimal design problems with $f \equiv 1$. Notice that the last row corresponds to the local problem. $\mathrm{DOFs} = \dim\mathcal{H}_h$, \textup{Its} = Number of Iterations.}
\label{tab:ConstRHSNL}
\end{table}

In the first set of experiments we set $f \equiv 1$. The results are presented in Figure~\ref{fig:ConstRHSNL} and Table~\ref{tab:ConstRHSNL}. In both the Figure and Table, the last row corresponds to the local problem. As we observe qualitatively in Figure~\ref{fig:ConstRHSNL}, and quantitatively, say, through the value of $r_h^s(\ofraka_{s,h})$ (last column of Table~\ref{tab:ConstRHSNL}) as we refine the mesh and increase the value of $s$ towards $1$ the solutions converge to one of the continuous local problem.

\begin{figure}
  \begin{center}
    \includegraphics[scale = 0.51]{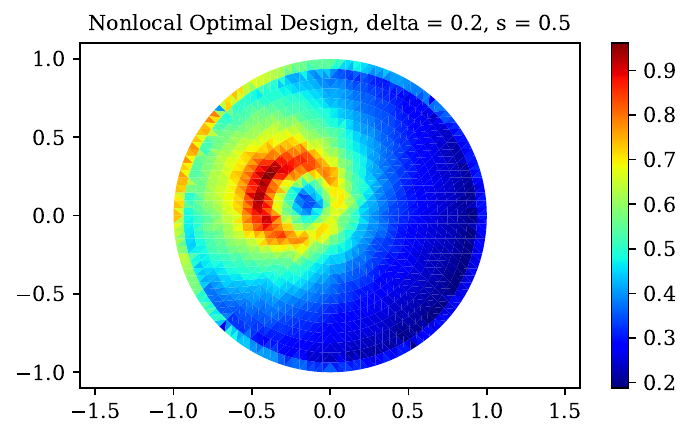}
    \includegraphics[scale = 0.5]{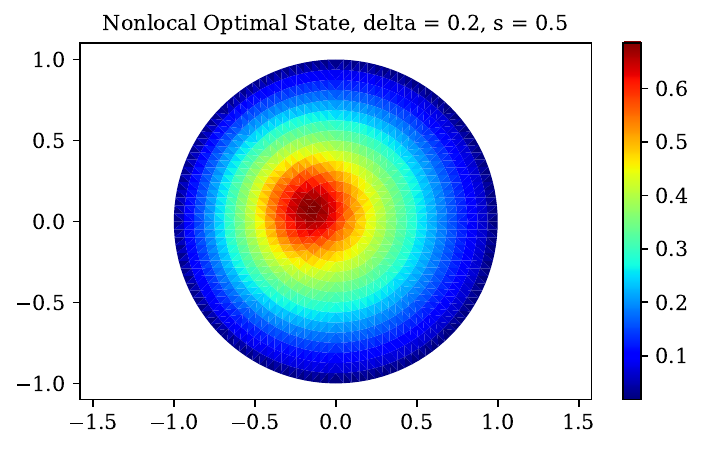}

    \includegraphics[scale = 0.51]{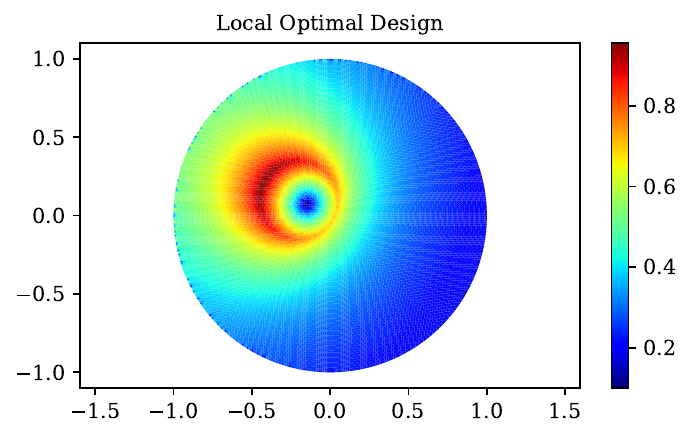}
    \includegraphics[scale = 0.5]{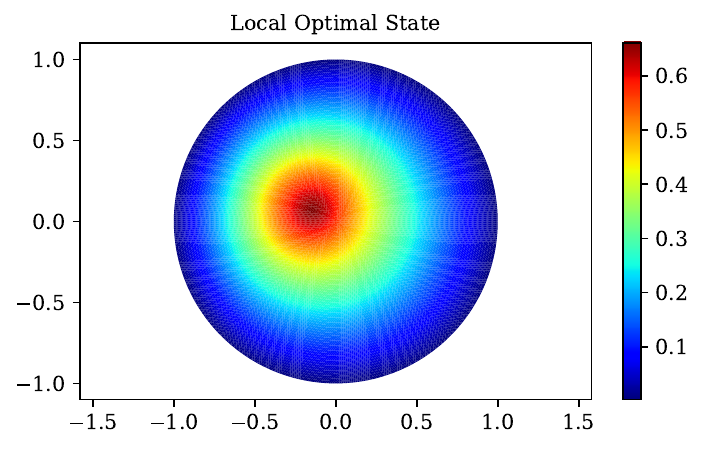}
  \end{center}
  \caption{Optimal design and state for $f = 3\chi_D$ with $D = B_{\tfrac14}(-0.2,0.1)$. The top row depicts the optimal design and state for $s=0.5$ and $R=0.2$. The bottom row corresponds to the local problem.}
  \label{fig:LocalizedRHSNL}
\end{figure}

\begin{table}
  \begin{center}
    \begin{tabular}{|| c | c | c | c | c | c | c | c||}
      \hline
      DOFs & Its & $s$ & $R$ & $\| \ou_{s,h} \|_{L^2(\Omega)}$ & $| \ou_{s,h} |_{\widetilde{H}^s(\Omega)}$ & $\| \ofraka_{s,h} \|_{L^2(\Omega)}$ & $r_h^{s}(\ofraka_{s, h})$ \\
      \hline\hline
      $961$   & $20$ & $0.5$ & $0.2$ & $0.499126$ & $1.670406$ & $0.833187$ &$0.694877$  \\
      $16129$ & $20$ & $1$   & $0$   & $0.453425$ & $1.604806$ & $0.807133$ & $0.652004$ \\
     \hline
    \end{tabular}
  \end{center}
\caption{Experimental data for optimal design problems with $f = 3\chi_D$ with $D = B_{\tfrac14}(-0.2,0.1)$. Notice that the last row corresponds to the local problem. $\mathrm{DOFs} = \dim\mathcal{H}_h$, \textup{Its} = Number of Iterations.}
\label{tab:LocalizedRHS}
\end{table}

As a second example we set $f = 3\chi_D$ with $D = B_{\tfrac14}(-0.2,0.1)$. Figure~\ref{fig:LocalizedRHSNL} and Table~\ref{tab:LocalizedRHS} show the results. The same type of convergence can be observed.

\section*{Acknowledgements}

TM  has been partially supported by NSF grants DMS-1910180 and DMS-2206252.
AJS has been partially supported by NSF grants DMS-2111228 and DMS-2409918.
JMS has been partially supported by NSF grant DMS-2111228.

The authors are grateful to Christian for answering questions pertaining to the development of the code.

\section*{Data availability statement}

All the data used for the numerical illustrations of this work were generated with an in-house developed code. The code is available as \cite{SiktarCode} and builds on the PyNucleus library \cite{PyNucleus}. PyNucleus is a Python code developed by Christian Glusa and funded by the LDRD program at Sandia National Laboratories.

\appendix

\section{A Meyers-type result for linear elasticity}
\label{sec:Appendix}

For completeness here we present, inspired by \cite[Section 2.2]{MR3129757}, a higher integrability result for the system of linear elasticity. We need this in Section~\ref{sec:Peridynamics} to prove properties of the Galerkin projection, and we believe that it may be of independent interest.

We begin by establishing some setup. Let $\Omega \subset \R^n$ be a bounded Lipschitz domain. For $p \in (1,\infty)$ we endow $\bW^{1,p}_0(\Omega)$ with the norm
\[
  \| \bw \|_{E,p} = \left( 2 \| \beps(\bw) \|_{\bL^p(\Omega)}^p + \| \DIV \bw \|_{L^p(\Omega)}^p \right)^{1/p}.
\]
The space $\bW^{-1,p}(\Omega)$ is then normed with the induced norm. Clearly, these norms are equivalent to the canonical ones.

We assume that the Lam\'e parameters $\lambda, \mu \in L^\infty(\Omega)$ satisfy
\[
  \lambda(x), \mu(x) \in [a,A], \quad \mae x \in \Omega,
\]
for some $0<a\leq A$.

We consider now the following problem: let $p \in (1,\infty)$ and $\bef \in \bW^{-1,p}(\Omega)$. Find $\bu \in \bW^{1,p}_0(\Omega)$ such that, in $\bW^{-1,p}(\Omega)$,
\begin{equation}
\label{eq:Elastic}
  -2\DIV( \mu \beps(\bu) ) - \GRAD( \lambda \DIV \bu ) = \bef.
\end{equation}
Our goal here shall be to show that there is $\varepsilon>0$ that depends only on the domain, $a$, and $A$, such that if $|p-2| < \varepsilon$, then this problem has a solution.

\noindent \textbf{The constant coefficient case.} Assume that $\lambda \equiv \mu \equiv 1$ and define the mapping $\bT: \bW^{1,p}_0(\Omega) \to \bW^{-1,p}(\Omega)$ as
\[
  \bT \bw = -2\DIV( \beps(\bw) ) - \GRAD( \DIV \bw ).
\]
If $p = 2$ this mapping is clearly invertible and, because of the way we have chosen the norms, $\| \bT^{-1} \|_2 = 1$. On the other hand, owing to \cite[Section 13.6]{MR2181934} and \cite[Theorem 14]{MR1781091} there is $\wp>2$, that depends only on the domain, for which the operator $\bT$ is boundedly invertible in $\bW^{1,\wp}_0(\Omega)$. Let $\| \bT^{-1} \|_\wp = K$. The Riesz-Thorin interpolation theorem \cite{MR482275} then shows that
\[
  \| \bT^{-1} \|_p \leq K^\theta, \qquad \frac1p = \frac{1-\theta}2 + \frac\theta\wp, \qquad p \in [2,\wp].
\]

\noindent \textbf{A perturbation argument.} Given Lam\'e parameters $\mu$ and $\lambda$ we define mappings $\bE,\bQ: \bW^{1,p}_0(\Omega) \to \bW^{-1,p}(\Omega)$ via
\[
  \bE\bw = -\frac2A\DIV( \mu \beps(\bw) ) - \frac1A\GRAD( \lambda \DIV \bw ), \qquad \bQ = \bT - \bE.
\]
Notice now that, for every $p \in (1,\infty)$, $ \| \bE \|_p \leq 1 $ and, with $q = \tfrac{p}{p-1}$,
\begin{multline*}
  \| \bQ \|_p =  \sup_{\bw \in \bW^{1,p}_0(\Omega)} \sup_{\bv \in \bW^{1,q}_0(\Omega)} \frac1{\| \bw \|_{E,p} \| \bv \|_{E,q}} \left[
    2 \int_\Omega \left( \frac\mu{A} -1 \right) \beps(\bw(x)):\beps(\bv(x)) \diff x + \right. \\
    \left. \int_\Omega \left(\frac\lambda{A} -1 \right) \DIV(\bw(x)) \DIV(\bv(x)) \diff x \right] \leq
    \left| \frac{a}A -1 \right|.
\end{multline*}

Next observe that, if $\bI$ is the identity operator,
\[
  \bE = \bT - \bQ = \bT\left( \bI - \bT^{-1} \bQ \right).
\]
This shows then that, if $p \in [2,\wp]$, the operator $\bE$ is invertible in $\bW^{1,p}_0(\Omega)$ provided
\[
  \| \bT^{-1} \|_p \| \bQ \|_p \leq K^\theta \left( 1- \frac{a}A \right) < 1, \qquad \frac1p = \frac{1-\theta}2 + \frac\theta\wp.
\]
Given $a$, $A$, and $\wp$ one can clearly chose a sufficiently small $\theta$ (i.e., $p$ is sufficiently close to $2$) so that the requisite inequality is satisfied.

As a consequence we have the following Corollary.

\begin{corollary}[higher integrability]
In the setting of Section~\ref{sec:Peridynamics} assume that $\Omega$ is Lipschitz, $\fraka \in \calH$, $\bef \in \bL^2(\Omega)$, and that $\bu \in \bH^1_0(\Omega)$ solves
\[
  \frakB_1[\fraka](\bu,\bv) = \int_\Omega \bef(x) \cdot \bv(x) \diff x, \qquad \forall \bv \in \bH^1_0(\Omega).
\]
Then, there is $\varepsilon>0$ that depends only on $\Omega$, $a_{\min}$, and $a_{\max}$ such that $\bu \in \bW^{1,2+\varepsilon}(\Omega)$, with a corresponding estimate.
\end{corollary}
\begin{proof}
  It suffices to realize that, for some $r>2$, $L^2(\Omega) \hookrightarrow W^{-1,r}(\Omega)$.
\end{proof}

\bibliographystyle{amsplain}
\bibliography{biblio}

\end{document}